\documentclass[sn-mathphys,Numbered]{sn-jnl}

\usepackage{graphicx}%
\usepackage{multirow}%
\usepackage{amsmath,amssymb,amsfonts}%
\usepackage{amsthm}%
\usepackage{mathrsfs}%
\usepackage[title]{appendix}%
\usepackage{xcolor}%
\usepackage{textcomp}%
\usepackage{manyfoot}%
\usepackage{booktabs}%
\usepackage{algorithm}%
\usepackage{algorithmicx}%
\usepackage{algpseudocode}%
\usepackage{listings}%
\usepackage{empheq}
\usepackage{subcaption,wrapfig}
\usepackage{tikz}
\usetikzlibrary{spy}
\usepackage{multirow, makecell,enumitem}

\newtheorem{theorem}{Theorem}

\newtheorem{remark}{Remark}%
\newtheorem{lemma}{Lemma}%
\newtheorem{proposition}{Proposition}%
\newtheorem{corollary}{Corollary}%
\newtheorem{definition}{Definition}%

\newcommand{\norm}[1]{\left|\left|{#1}\right|\right|}
\newcommand{\prox}{\operatorname{Prox}}
\newcommand{\rprox}{\operatorname{Rprox}}
\newcommand{\id}{\operatorname{Id}}
\newcommand{\dom}{\operatorname{dom}}
\renewcommand{\Im}{\operatorname{Im}}
\newcommand{\Fix}{\operatorname{Fix}}

\newcommand{\dime}{n}

\DeclareMathOperator*{\argmin}{arg\,min}
\newcommand{\reglambda }{\lambda}
\newcommand{\R}{\mathbb{R}}

\newcommand{\zeros}{\operatorname{zeros}}

\begin{document}

\title[Convergent plug-and-play with proximal denoiser and unconstrained regularization parameter]{Convergent plug-and-play with proximal denoiser and unconstrained regularization parameter}


\author[1]{\fnm{Samuel} \sur{Hurault}}\email{samuel.hurault@math.u-bordeaux.fr}

\author[2]{\fnm{Antonin} \sur{Chambolle}}\email{antonin.chambolle@ceremade.dauphine.fr}

\author[1]{\fnm{Arthur} \sur{Leclaire}}\email{arthur.leclaire@math.cnrs.fr}

\author[1]{\fnm{Nicolas} \sur{Papadakis}}\email{nicolas.papadakis@math.u-bordeaux.fr}

\affil*[1]{\orgdiv{Univ. Bordeaux, CNRS, Bordeaux INP, IMB, UMR 5251}, 
\orgaddress{
\postcode{F-33400}, \state{Talence}, \country{France}}}

\affil[2]{\orgdiv{CEREMADE, CNRS, Universit\'e Paris-Dauphine, PSL}, 
\orgaddress{
\city{Palaiseau}, 
\postcode{91128}, \country{France}}}


\abstract{
In this work, we present new proofs of convergence for Plug-and-Play (PnP) algorithms. PnP methods are efficient iterative algorithms for solving image inverse problems where regularization is performed by plugging a pre-trained denoiser in a proximal algorithm, such as Proximal Gradient Descent (PGD) or Douglas-Rachford Splitting (DRS). Recent research has explored convergence by incorporating a denoiser that writes exactly as a proximal operator. However, the corresponding PnP algorithm has then to be run with stepsize equal to $1$. The stepsize condition for nonconvex convergence of the proximal algorithm in use then translates to restrictive conditions on the regularization parameter of the inverse problem. This can severely degrade the restoration capacity of the algorithm. In this paper, we present two remedies for this limitation. First, we provide a novel convergence proof for PnP-DRS that does not impose any restrictions on the regularization parameter. Second, we examine a relaxed version of the PGD algorithm that converges across a broader range of regularization parameters. Our experimental study, conducted on deblurring and super-resolution experiments, demonstrate that both of these solutions enhance the accuracy of image restoration.
}




\maketitle

\section{Introduction}\label{sec1}

In this work, we study the convergence of Plug-and-Play algorithms for solving image inverse problems. This is achieved by minimizing an explicit function:
\begin{equation} \label{eq:pb}
\hat x \in \argmin_{x} \lambda f(x) + \phi(x) .
\end{equation}
The function \( f \) represents a data-fidelity term that quantifies the alignment of the estimation \( x \) with the degraded observation \( y \), the function \( \phi \) is a \textit{nonconvex} regularization function, and \( \lambda > 0 \) acts as a parameter, determining the balance between these two terms.

In our experimental setup, we focus on degradation models where \( y \) is expressed as \( y = Ax^* + \nu \), with \( x^* \) representing the ground-truth signal in \(\mathbb{R}^n\), \( A \) a linear operator in \(\mathbb{R}^{n\times m}\), and \( \nu \) a white Gaussian noise in \(\mathbb{R}^n\). In this context, the data-fidelity term \( f(x) \) takes the form of a convex and smooth function \( f(x) = \frac12 \lVert Ax - y \rVert^2 \). However, our theoretical analysis extends to a wider range of noise models, as well as data-fidelity terms \( f \) that may not be smooth or convex.

To find an adequate solution of the ill-posed problem of recovering $x^*$ from~$y$, the choice of the regularization $\phi$ is crucial. Research has first been dedicated to find, by hand, regularizing potentials $\phi$ that are minimal when a desired property of a clean image is satisfied. Among the most classical priors, one can single out total variation~\cite{ROF}, $L^2$ gradient norm (Tikhonov)~\cite{tikhonov1963solution}, or sparsity in a given dictionary, for example in the wavelet representation~\cite{mallat2009sparse}. These handcrafted priors are now largely outperformed by learning approaches~\cite{zoran2011learning, bora2017compressed, romano2017little, zhang2021plug}, which may not even be associated to a closed-form regularization function $\phi$. 

Estimating a local or global optima of problem~\eqref{eq:pb} is classically done using proximal splitting optimization algorithms such as Proximal Gradient Descent~(PGD), Douglas-Rashford Splitting (DRS) or Primal-Dual. Given \hbox{an adequate stepsize $\tau>0$,} these methods alternate between explicit gradient descent steps, $\id-\tau \nabla h$ for smooth functions $h$, and/or \hbox{implicit gradient steps} using the proximal operator
${\prox_{\tau h}(x) \in \argmin_z \frac1{2\tau}||z-x||^2+h(z)}$, for a proper lower semi-continuous function~$h$.
Proximal algorithms are originally designed for convex functions, but under appropriate assumptions, PGD~\cite{attouch2013convergence} and DRS~\cite{themelis2020douglas}  algorithms converge to a stationary point of problem~\eqref{eq:pb} associated to nonconvex functions $f$ and $\phi$.

\paragraph{Plug-and-Play algorithms} 
Plug-and-Play (PnP)~\cite{venkatakrishnan2013plug} and Regularization by Denoising (RED)~\cite{romano2017little} methods consist in splitting algorithms in which the descent step on the regularization function is performed by an off-the-shelf image denoiser. They are respectively built from proximal splitting schemes by replacing the proximal operator (PnP) or the gradient operator (RED) of the regularization $\phi$ by an image denoiser. 
When used with a deep denoiser (\emph{i.e} parameterized by a neural network) these approaches produce impressive  results for various image restoration tasks~\cite{meinhardt2017learning, zhang2017learning, sun2019block, ahmad2020pnp, Yuan_2020_CVPR, sun2021scalable, zhang2021plug}. 

Theoretical convergence of PnP and RED algorithms with deep denoisers has recently been addressed by a variety of studies~\cite{ryu2019plug,sun2021scalable,pesquet2021learning,hertrich2021convolutional}. Most of these works require specific constraints on the deep denoiser, such as nonexpansivity. However, imposing nonexpansivity of a denoiser can severely degrade its performance (\cite[Table 1]{hertrich2021convolutional}, 
\cite[Figure 1]{bohra2021learning}, \cite[Table 3]{nair2022construction}).

Another line of works \citep{sreehari2016plug, cohen2021has, hurault2021gradient, hurault2022proximal, hurault2023relaxed} tries to address convergence by making PnP and RED algorithms exact proximal algorithms. The idea is to replace the denoiser of RED algorithms by a gradient descent operator and the one of PnP algorithm by a proximal operator. Theoretical convergence then follows from known convergence results of proximal splitting algorithms. In~\cite{cohen2021has,hurault2021gradient}, it is thus proposed to plug an explicit \emph{gradient-step denoiser} of the form $D=\id-\nabla g$, for a tractable and potentially nonconvex potential $g$ parameterized by a neural network. 
As shown in~\cite{hurault2021gradient}, such a constrained parametrization does not harm denoising performance. The gradient-step denoiser guarantees convergence of RED methods without sacrificing performance, but it does not cover convergence of PnP algorithms. 
An extension to PnP has been addressed in~\cite{hurault2022proximal, hurault2023relaxed}: following~\cite{gribonval2020characterization}, when $g$ is trained with nonexpansive gradient, the gradient-step denoiser can be written as a proximal operator $D=\id-\nabla g=\prox_{\phi}$ of a nonconvex potential $\phi$. A PnP scheme with this \emph{proximal denoiser} becomes again a genuine proximal splitting algorithm associated to an explicit nonconvex functional. Following existing convergence results of the PGD and DRS algorithms in the nonconvex setting, \cite{hurault2022proximal} proves convergence of PnP-PGD and PnP-DRS with proximal denoiser (called ProxPnP-PGD and ProxPnP-DRS). 

The main limitation of this approach is that the proximal denoiser $D=\prox_{\phi}$ does not give tractability of $\prox_{\tau \phi}$ for $\tau \neq 1$. Therefore, to be a provable converging proximal splitting algorithm, the stepsize of the overall PnP algorithm has to be fixed to $\tau=1$. 
For instance, when used with stepsize $\tau=1$, both the PGD algorithm \cite[Theorem 4.1]{hurault2022proximal} 
\begin{equation}
    \label{eq:PGD0}
    x_{k+1} \in \prox_{\phi} \circ (\id - \lambda \nabla f)(x_k)
\end{equation}
and the DRS algorithm \cite[Theorem 4.3]{hurault2022proximal} 
\begin{equation}
    x_{k+1} \in \left( \frac{1}{2}\rprox_{ \lambda f} \circ \rprox_{ \phi_\sigma} + \frac{1}{2}\id \right) (x_k)
\end{equation}
are proved to converge to a stationary point of~\eqref{eq:pb} for a regularization parameter $\lambda$ satisfying ${\lambda L_f<1}$. This is an issue when the the input image has mild degradations, for instance for low noise levels $\nu$, because in this case relevant solutions are obtained with a dominant data-fidelity term in~\eqref{eq:pb} through high values $\lambda>1 / L_f$. Note that convergence results for the DRS algorithm without such restriction on $\lambda$ has also been proposed in  \cite[Theorem 4.4]{hurault2022proximal} but at the cost of additional technical assumptions on the geometry of the denoiser which are difficult to verify in practice.

In this work, our objective is to design convergent PnP algorithms with a proximal denoiser, and with minimal restriction on the regularization parameter $\lambda$. Contrary to previous works on PnP convergence~\cite{ryu2019plug,sun2021scalable,terris2020building,cohen2021has,hurault2021gradient,hertrich2021convolutional}, we not only wish to adapt the denoiser but also the original optimization scheme of interest. We show and make use of the fact that the regularization $\phi$ (from $D = \id - \nabla g = \prox_{\phi}$) is weakly convex, with a weak convexity constant that depends only on the Lipschitz constant of the gradient of the learned potential $g$. We first recall the main result from \cite{hurault2023relaxed} that studies a new proximal algorithm, reminiscent to PGD, which converges for a fixed stepsize $\tau=1$ with relaxed restriction on $\lambda$. Second, we adapt the convergence study from \cite{themelis2020douglas} of the DRS algorithm in the nonconvex setting and prove convergence without restriction on $\lambda$. With both results, the constraint on $\lambda$ is replaced by a constraint on the Lipschitz constant of the gradient of the learned potential $g$. This new condition can be naturally satisfied using a convex relaxation of the gradient-step denoiser.

\subsection*{Contributions and outline} 

This work is an extension from \cite{hurault2023relaxed}. After  detailing the \emph{$\alpha$PGD} relaxation of the Proximal Gradient Descent algorithm introduced in  \cite{hurault2023relaxed}, we propose a new proof of convergence of the DRS algorithm, such that when used with a proximal denoiser, the corresponding PnP schemes~\emph{ProxPnP-DRS} can converge for any regularization parameter $\lambda$. 

In section~\ref{sec:defs_prop}, we start by introducing a variety of concepts and properties that will be used in our convergence analysis.   

In section~\ref{sec2}, we explain how the gradient-step denoiser $D = \id - \nabla g$ introduced in \cite{hurault2021gradient} can write as the proximal operator of a $M$-weakly convex function $\phi$. This denoiser is then called \emph{proximal denoiser}.
The weak convexity constant $M$ depends only on the Lipschitz constant of $\nabla g$ and can be controlled with a convex relaxation of the denoiser.

In section~\ref{sec3}, we focus on the nonconvex convergence of ProxPnP-PGD \cite{hurault2022proximal} \emph{i.e.} the PnP version of the standard PGD algorithm, with plugged proximal denoiser. By exploiting the weak convexity of $\phi$, the condition for convergence is $\lambda L_f+M <2$. 

In section~\ref{sec4}, we detail the \emph{$\alpha$PGD}, a relaxed version of the PGD algorithm\footnote{There are two different notions of relaxation in this paper. One is for the relaxation of the proximal denoiser and the other for the relaxation of the optimization scheme.} introduced in \cite{hurault2023relaxed}. Its convergence is shown in  Theorem~\ref{thm:alphaPGD} for a smooth convex function~$f$ and a weakly convex one $\phi$. Corollary~\ref{cor:ProxPnP-alphaPGD} then applies this result for proving convergence of ProxPnP-$\alpha$PGD, its PnP version with proximal denoiser, under the condition $\lambda L_f M <1$. Having a multiplication, instead of an addition, between the constants $\lambda L_f$ and $M$ allows to use any regularization parameter $\lambda$, provided we decrease $M$ sufficiently.

In section~\ref{sec:DRS}, we propose a new convergence proof for the PnP-DRS algorithm with proximal denoiser, without any restriction on the regularization parameter $\lambda$.  We extend the original nonconvex convergence result from \cite{themelis2020douglas} and prove that, under a new constraint on the weak convexity constant $M$, the PnP version of the DRS algorithm with proximal denoiser also converges to a stationary point of an explicit functional.

In section~\ref{sec:expes}, we realize experiments for both image deblurring and image super-resolution applications. We demonstrate that, by eliminating the constraint on the regularization parameter $\lambda$, our PnP-$\alpha$PGD and PnP-DRS algorithms close the performance gap with state-of-the-art plug-and-play algorithms.

\section{Definitions and useful properties}\label{sec:defs_prop}

\subsection{Definitions} 

In this section, we consider $f : \mathbb{R}^n \to \mathbb{R} \cup \{ + \infty \}$ proper and lower semicontinuous.
We recall that $f$ is \textit{proper} if  $\dom(f) \neq \emptyset$ and \textit{lower semicontinuous (lsc) } if $\forall x\in\mathbb{R}^n$ $\underset{y\to x}{\lim\inf}f(y)\geqslant f(x)$.
For $M \in \R$, $f$ is called $M$-weakly convex if $ x \to f(x)+ \frac{M}{2}\norm{x}^2$ is convex. Moreover, it is \textit{coercive} if ${\underset{\norm{x}\to+\infty}{\lim}f(x)=+\infty}$.

\paragraph{Nonconvex subdifferential} When $f$ is nonconvex, the standard notion of subdifferential 
\begin{equation}
\begin{split}
    \partial f (x) = \{ \omega &\in \mathbb{R}^n, \forall y \in \mathbb{R}^n, 
    f(y) - f(x) - \langle \omega, y-x \rangle \geq 0 \}
    \end{split}
\end{equation}
may not be informative enough. Indeed, with Fermat's rule, in both the convex and nonconvex cases, $\operatorname{Argmin} f = \zeros(\partial f)$, but for minimizing a nonconvex function, there is usually no hope to target a global minimum. Following \cite{attouch2013convergence}, we will use as notion of subdifferential the \textit{limiting subdifferential}
\begin{equation} \label{eq:limiting_subdifferential}
\begin{split}
       \partial^{lim} f(x) = \{&\omega \in \mathbb{R}^n, \exists x_k \to x, f(x_k) \rightarrow f(x), \omega_k \rightarrow \omega, \omega_k \in \hat \partial f(x_k) \}
\end{split}
\end{equation}
with $\hat \partial f$ the  \textit{Fréchet subdifferential} of $f$ defined as
\begin{equation}
\begin{split}
    \hat \partial f(x) &= \{ \omega \in \mathbb{R}^n, \liminf_{y \to x} \frac{f(y) - f(x) - \langle \omega, y-x \rangle }{\norm{x-y}} \geq 0 \} .
\end{split}
 \end{equation} These three notions of subdifferential verify for $x \in \dom f$
\begin{equation} \label{eq:inclusion_subdiff}
    \partial f(x) \subset \hat \partial f(x) \subset \partial^{lim} f(x) .
\end{equation}
and coincide when $f$ is convex.

A necessary (but not sufficient) condition for $x \in \R^n$ to be a local minimizer of a nonconvex function $f$ is  $0 \in \hat \partial f(x)$ and thus  $0 \in \partial^{lim}  f(x) $. 
Last but not least, we give two important properties of the limiting subdifferential.
\begin{proposition}[Sum rule {\cite[8.8(c)]{rt1998wets}}] \label{prop:subdif_sum}
    If $F = f + g$ with $f$ of class $\mathcal{C}^1$ and $g$ proper. Then for $x \in \dom(g)$, 
    \begin{equation}
        \partial^{lim} F(x) = \nabla F(x) + \partial^{lim} g(x).
    \end{equation}
\end{proposition}
\begin{proposition}[Closedness of the limiting subdifferential]
\label{prop:subdif_closed}
   For a sequence $(x_k, \omega_k) \in \operatorname{Graph}(\partial^{lim} f)$, if $(x_k, \omega_k) \to (x,\omega)$ and $f(x_k) \to f(x)$, then  $(x, \omega) \in \operatorname{Graph} (\partial^{lim} f)$,
\end{proposition}
\noindent where $\operatorname{Graph}(F)$ stands for the graph of a point-to-set mapping
\begin{equation}
    \operatorname{Graph} F = \{ (x,y) \in \R^n \times \R^m, y \in F(x) \} .
\end{equation}

For additional details on the notion of limiting subdifferential, we refer to \citep[Chapter 8]{rt1998wets}.
In the rest of the paper, the subdifferential $\partial f$ corresponds to the limiting subdifferential $\partial^{lim}$ dfined in~\eqref{eq:limiting_subdifferential}.

\paragraph{Kurdyka–Łojasiewicz property}

The Kurdyka-Łojasiewicz (KŁ) property is a local property that characterizes the shape of the function $f$ around its critical points $\{x \in int \dom(f), 0 \in \partial f(x) \}$.
It is a tool widely used in nonconvex optimization \citep{attouch2010proximal, attouch2013convergence, ochs2014ipiano, bolte2018first, zeng2019global}. We use the definition from~\cite{attouch2010proximal}:
\begin{definition}[Kurdyka-Łojasiewicz (KŁ) property]
    \label{def:KŁ}
        A function $f : \mathbb{R}^n \to \mathbb{R} \cup \{+\infty \}$ is said to have the Kurdyka-Łojasiewicz property at $x^* \in dom(f)$ if there exists $\eta \in (0,+\infty)$, a neighborhood $U$ of $x^*$ and a continuous concave function $\psi : [0,\eta) \to \mathbb{R}_+$ such that
        $\psi(0)=0$, $\psi$ is $\mathcal{C}^ 1$ on $(0,\eta)$, $\psi' > 0$ on $(0,\eta)$ and $\forall x \in U \cap [f(x^*) < f < f(x^*)+\eta]$, 
        the \emph{Kurdyka-Łojasiewicz inequality} holds:
        \begin{equation} \label{eq:KL_ineq}
            \psi'(f(x)-f(x^*))\text{dist}(0,\partial f(x)) \geq 1 .
        \end{equation}
        Proper lsc functions that satisfy the Kurdyka-Łojasiewicz inequality at each point of $dom(\partial f)$ are called KŁ functions.
\end{definition}
\begin{remark}
    \begin{itemize}
        \item[(i)] For $f$ proper lsc, the KŁ inequality always holds at non-critical points $x^* \in \dom \partial f$ \cite[Remark 4.(b)]{attouch2010proximal}.  
        \item[(ii)] For $f$ smooth and $f(x^*)=0$, the KŁ inequality~\eqref{eq:KL_ineq} writes $\norm{\nabla (\psi \circ f)(x)} \geq 1$. The KŁ property can thus be interpreted as the fact that, up to a reparameterization, the function is locally sharp \citep{attouch2013convergence}. 
    \end{itemize}
\end{remark}
Folllowing \citep[Theorem 2.9]{attouch2013convergence}, to prove  the single-point convergence of iterative algorithms for minimizing the function $F = f + g$, with $f$ or $g$ nonconvex, the first requirement is to show that $F$ is KŁ. However, the KŁ condition is not stable by sum. For the purpose of the present paper, we need to choose a subclass of KŁ functions that, on the one hand, is large enough to encompass all our functions of interest and, on the other hand, has minimal stability properties so that inclusion to that set is easy to verify. In particular, as some of our functions will be parameterized by neural networks, stability by sum and composition are minimal requirements.  
In this analysis, the set of \textit{real analytic functions} is large enough to include all our functions of interest. 
\begin{definition}[Real analytic] \label{def:analytic}~
A function  $f : \mathbb{R}^n \to \mathbb{R} \cup \{+\infty \}$ with open domain
is said to be (real) analytic at $x \in \dom(f)$ 
if $f$ may be represented by a convergent power series on a neighborhood of $x$. The function is said to be analytic if it is analytic at each point of its domain. 
We can extend the definition for $f : \mathbb{R}^n \to \mathbb{R}^m$ which is analytic at $x \in \R^n$ if for $f = (f_1, \ldots, f_m)$, all $f_i : \mathbb{R}^n \to \R$ are analytic at $x$.
\end{definition}
Typical real analytic functions include polynomials, exponential functions, the logarithm, trigonometric and power functions. Moreover, real analytic functions have the following stability properties
\begin{lemma}[\cite{krantz2002primer}] \label{lem:stability-analytic}
    \begin{itemize}
    \item[(i)] The sum, product, and composition of real analytic functions are real analytic.
    \item[(ii)] A partial derivative of $f$ real analytic is real analytic.
    \item[(iii)] \underline{Real Analytic Inverse Function Theorem: } Let $f : \R^n \to \R^m$ be real analytic in a neighborhood of $x \in \R^n$ and suppose $J_f(x)$ non-singular, then $f^{-1}$ is defined and is real analytic in a neighborhood of $f(x)$. In particular, if $f$ is real analytic and $\forall x \in \R^n$, $J_f(x)$ is non-singular, then $f^{-1}$ is real analytic on $Im(f)$. 
\end{itemize}
\end{lemma}
Finally, the important point here is that real analytic functions are KŁ: 
\begin{lemma}[\cite{law1965ensembles}]
    Let $f : \mathbb{R}^n \to \mathbb{R} \cup \{+\infty \}$ be a proper real analytic function, then for any critical point $x^* \in \dom(f)$, there exists a neighborhood $U$ of $x^*$, an exponent $\theta \in [\frac12, 1)$ and a constant $C$ such that the following \emph{Łojasiewicz inequality} holds:
        \begin{equation} \label{eq:KL_ineq2}
           \forall x \in U, \ \ |f(x)-f(x^*)|^\theta \leq C \norm{\nabla f(x)} .
        \end{equation}
\end{lemma}
\begin{remark}
    The Łojasiewicz inequality is a particular case of the Kurdyka-Łojasiewicz inequality from Definition~\ref{def:KŁ} with desingularizing function $\psi(s) = c s^{1-\theta}$. 
\end{remark}

\subsection{Useful inequalities}\label{sec:pgd_prop}
In this section, we give different inequalities verified by weakly convex and smooth functions. These results constitute basic blocks that will be used in our different proofs of convergence.

\begin{proposition}[Inequalities for weakly convex functions]
\label{prop:weaklyconvex}
For $f : \mathbb{R}^n \to \mathbb{R} \cup \{+\infty \}$ proper lsc and $M$-weakly convex with $M \in \R$,
\begin{itemize}
    \item[(i)]  $\forall x,y$ and $t \in [0,1]$, 
    \begin{equation} 
    \begin{split}
    f(tx+(1-t)y) \leq & tf(x) + (1-t)f(y) + \frac{M}{2}t(1-t)\norm{x-y}^2;
    \end{split}
    \end{equation}
    \item[(ii)] $\forall x,y$, we have $\forall z \in \partial f(y)$,
    \begin{equation}
    f(x) \geq f(y) + \langle z,x-y \rangle - \frac{M}{2}\norm{x-y}^2;
\end{equation}
 \item[(iii)] \textbf{Three-points inequality}. 
For $z^+ \in \prox_f(z) $, we have, $\forall x$
\begin{equation}
    f(x) + \frac{1}{2}\norm{x-z}^2 \geq f(z^+) +  \frac{1}{2}\norm{z^+-z}^2 + \frac{1-M}{2}\norm{x-z^+}^2.
\end{equation}
\end{itemize}
\end{proposition}
\begin{proof}
(i) and (ii) follow from the fact that $\phi + \frac{M}{2}\norm{x}^2$ is convex. We now prove (iii).
Optimality conditions of the proximal operator $z^+\in\prox_\phi(z)$ gives 
\begin{equation}
    z - z^+ \in \partial \phi(z^+).
\end{equation}
Hence, by (ii), we have $\forall x$,
\begin{equation}
\phi(x) \geq \phi(z^+) + \langle  z - z^+, x-z^+ \rangle - \frac{M}{2}\norm{x-z^+}^2,
\end{equation}
and therefore, 
\begin{equation}
\begin{split}
    \phi(x) + \frac{1}{2}\norm{x-z}^2
    &\geq \phi(z^+) + \frac{1}{2}\norm{x-z}^2 + \langle z - z^+, x-z^+ \rangle - \frac{M}{2}\norm{x-z^+}^2 \\
    &= \phi(z^+) + \frac{1}{2}\norm{x-z^+}^2 + \frac{1}{2}\norm{z-z^+}^2 - \frac{M}{2}\norm{x-z^+}^2 \\
    &= \phi(z^+)  + \frac{1}{2}\norm{z-z^+}^2 + \frac{1-M}{2}\norm{x-z^+}^2.
\end{split}
\end{equation}
\end{proof}


\begin{proposition}[Descent Lemma {\citep[Lemma 2.64]{BauschkeCombettes}}] \label{prop:descent_lemma}
        For $f : \mathbb{R}^n \to \mathbb{R} \cup \{+\infty \}$ proper lsc, if $f$ is $L_f$-smooth on an open and convex set ${C \subseteq int \dom(f)}$,  $\forall (x,y) \in C \times C$, we have
        \begin{equation}
        | f(x) - f(y) - \langle \nabla f(y), x-y \rangle | \leq \frac{L_f}{2}\norm{y-x}^2.
        \end{equation} 
    \end{proposition}
    

\section{Relaxed Proximal Denoiser}\label{sec2}

This section introduces the denoiser used in our PnP algorithms. 
We first recall the definition of the Gradient Step Denoiser and show in Proposition~\ref{prop:proxdenoiser} how it can be constrained to be a proximal denoiser. We finally introduce the relaxed proximal denoiser.

\subsection{Gradient Step  Denoiser}

In this paper, we make use of the Gradient Step Denoiser introduced in~\cite{hurault2021gradient,cohen2021has}, that writes as a gradient step over a differentiable potential~$g_\sigma$ parameterized by a neural network: 
\begin{equation}
    \label{eq:gs_denoiser}
    D_\sigma = \id - \nabla g_\sigma.
\end{equation}
 In~\cite{hurault2021gradient}, $g_\sigma$ is chosen to be of the form $g_\sigma(x)=\frac12||x-N_\sigma(x)||^2$ with $N_\sigma$ parameterized with a DRUNet architecture~\cite{zhang2021plug}. This denoiser can then be trained to denoise white Gaussian noise $\nu_\sigma$ of various standard deviations $\sigma$ by minimizing the $\ell_2$ denoising loss $\mathbf{E}[||D_\sigma(x+\nu_\sigma)-x)||^2]$.
 It is shown that the Gradient Step Denoiser~\eqref{eq:gs_denoiser}, despite being constrained to be a conservative vector field (as in~\cite{romano2017little}), achieves state-of-the-art denoising performance.

\subsection{Proximal Denoiser}\label{ssec:prox_denoiser}
We first present a characterization of the Gradient Step denoiser as a proximal operator of some weakly convex potential $\phi$. 
The proof of this result relies on~\cite{gribonval2020characterization}. 
\begin{proposition}[Proximal Gradient Step denoiser]
   \label{prop:proxdenoiser} Let ${g_\sigma : \mathbb{R}^n \to \mathbb{R}}$ a $\mathcal{C}^{k+1}$ function with $k \geq 1$ and $\nabla g_\sigma$ $L_{g_\sigma}$-Lipschitz with $L_{g_\sigma}<1$. Let ${D_\sigma := \id - \nabla g_\sigma = \nabla h_\sigma}$. 
    Then,
    \begin{itemize}
    \item[(i)] there is ${\phi_\sigma:\mathbb{R}^n  \to  \mathbb{R}^n \cup \{+\infty\}}$,  {$\frac{L_{g_\sigma}}{L_{g_\sigma}+1}$-weakly convex}, such that $\prox_{\phi_\sigma}$ is one-to-one and 
    \begin{equation}
    \label{eq:proximal_denoiser}
        D_\sigma = \prox_{\phi_\sigma}.
    \end{equation}
    Moreover $D_\sigma$ is injective, $\Im(D_\sigma)$ is open and there is a constant $K \in \R$ such that $\phi_\sigma$ is defined on $\Im(D_\sigma)$ by $\forall x \in \Im(D_\sigma)$,
    \begin{equation} \label{eq:phi}
    \begin{split}
        \phi_\sigma(x) &= g_\sigma(D_\sigma^{-1}(x))-\frac{1}{2} \norm{D_\sigma^{-1}(x)-x}^2 + K \\
        &= h_\sigma^{*}(x) - \frac{1}{2}\norm{x}^2 + K.
    \end{split}
    \end{equation}
    where $h_\sigma^{*}$ stands for the convex conjugate of $h_\sigma$.
        \item[(ii)] $\forall x \in \mathbb{R}^n$, $\phi_\sigma(x) \geq g_\sigma(x) + K$ and for $x \in \Fix(D_\sigma)$, $\phi_\sigma(x) = g_\sigma(x) + K$.
        \item[(iii)] $\phi_\sigma$ is $\mathcal{C}^k$ on $\Im(D_\sigma)$ and $\forall x \in \Im(D_\sigma)$, ${\nabla \phi_\sigma(x) = {D_\sigma}^{-1}(x) - x = \nabla g_\sigma ({D_\sigma}^{-1}(x))}$. Moreover, $\nabla \phi_\sigma$ is $\frac{L_{g_\sigma}}{1-L_{g_\sigma}}$-Lipschitz on $\Im(D_\sigma)$
    \end{itemize}
\end{proposition}

\begin{proof}
\textbf{(i)} This result is an extension of the Proposition 3.1 from \cite{hurault2022proximal}. Under the same assumptions as ours, it is shown, using the characterization of the proximity operator from \cite{gribonval2020characterization}, that for $\hat \phi_\sigma : \mathbb{R}^n \to \mathbb{R}$ defined as 
\begin{equation}\hspace*{-.5cm}
    \hat \phi_\sigma(x)\hspace*{-1pt}:=\hspace*{-1pt}\left\{\begin{array}{ll}   g_\sigma({D_\sigma}^{-1}(x)))-\frac{1}{2} \norm{{D_\sigma}^{-1}(x)-x}^2  \text{   if }\ x \in \Im(D_\sigma),\\
     +\infty \text{    otherwise} \end{array}\right.
\end{equation} 
we have $D_\sigma = \prox_{\hat \phi_\sigma}$.

As $\dom(\hat \phi_\sigma) = \Im(D_\sigma)$ is non necessarily convex, we cannot show that $\hat \phi_\sigma$ is weakly convex. We thus propose another choice of function  $\phi_\sigma$ such that $D_\sigma = \prox_{\phi_\sigma}$.
As we suppose that $\nabla g_\sigma  = \id - D_\sigma$ is $L_{g_\sigma}$ Lipschitz, $D_\sigma$ is $L_{g_\sigma}+1$ Lipschitz. By~\cite[Proposition 2]{gribonval2020characterization}, we have $\forall x \in \mathcal{X}$,
    \begin{equation}
    D_\sigma(x) \in \prox_{\phi_\sigma}(x)
    \end{equation}
with ${\phi_\sigma:\mathbb{R}^n  \to  \mathbb{R}^n \cup \{+\infty\}}$ which is 
$\left( 1 - \frac{1}{L+1}\right) = \frac{L}{L+1}$ weakly convex.
 The weak convexity constant being smaller than $1$, $x \to \phi_\sigma(x) + \frac{1}{2}\norm{x-y}^2$ is convex and $\prox_{\phi_\sigma}$ is one-to-one.
    According to~\cite[Theorem 4 (b)]{gribonval2020characterization}, there exists a constant $K \in \R$ such that for any $C \subseteq Im(D_\sigma)$ polygonally connected $\phi_\sigma = \hat \phi_\sigma + K$ on $C$ \emph{i.e.}
    \begin{align} \label{eq:phi-before-inv}
        \forall y \in C, \ \  \phi_\sigma(D_\sigma(y)) &= \langle y, D_\sigma(y) \rangle - \frac{1}{2}\norm{D_\sigma(y)}^2 - h_\sigma(y) + K  \\
        &= g_\sigma(y)-\frac{1}{2} \norm{y-D_\sigma(y)}^2 + K  \label{eq:phi-before-inv2}
    \end{align}
    Moreover, from the equality case of the Fenchel-Young inequality,
    \begin{equation} \label{eq:fy2}
        \langle y, D_\sigma(y) \rangle = \langle y, \nabla h_\sigma(y) \rangle = h_\sigma(y) + h_\sigma^*(D_\sigma(y)) .
    \end{equation}
    Combining~\eqref{eq:phi-before-inv}  and \eqref{eq:fy2}, we get
    \begin{equation}
    \phi_\sigma(D_\sigma(y)) = h_\sigma^*(D_\sigma(y)) - \frac{1}{2}\norm{D_\sigma(y)}^2  + K
    \end{equation}
    As $\nabla g_\sigma$ is $\mathcal{C}^2$ and $L<1$-Lipschitz, $\forall x \in \R^n, J_{D_\sigma}(x) = \id - \nabla^2 g_\sigma(x)$ is positive definite and $D_\sigma$ is injective. We can consider its inverse $D_\sigma^{-1}$ on $\Im(D_\sigma)$, and we get~\eqref{eq:phi} from~\eqref{eq:phi-before-inv2} 
    Also, the inverse function theorem ensures that $\Im(D_\sigma)$ is open in $\mathbb{R}^n$.
    As $D_{\sigma}$ is continuous, $\Im(D_\sigma)$ is connected and open, thus polygonally connected. Therefore, $\eqref{eq:phi-before-inv}$ is true on the whole $\Im(D_\sigma)$.

    \textbf{(ii)} We have
    \begin{equation}
    \begin{split}
            \phi_\sigma(x) &= \frac{1}{2}\norm{x-x}^2 + \phi_\sigma(x) \\
            &\geq \frac{1}{2}\norm{x-D_\sigma(x)}^2 + \phi_\sigma(D_\sigma(x)) \\
            &= g_\sigma(x) + K
    \end{split}
    \end{equation}
    where the first inequality comes from the definition of the proximal operator  $D_\sigma = \prox_{\phi_\sigma}$ and the last equality is given
    by~\eqref{eq:phi}.

    \textbf{(iii)} This is the application of \cite[Corollary 6]{gribonval2020characterization}. 
    For the Lipschitz property, let $x,y \in \Im(D_\sigma)$, there exists $u,v \in \mathbb{R}^n$ such that $x=D_\sigma(u)$ and $y = D_\sigma(v)$. Hence, we have
    \begin{equation}
    \begin{split}
        \norm{\nabla \phi_\sigma(x)- \nabla \phi_\sigma(y)}
        &= \norm{{D_\sigma}^{-1}(x)-{D_\sigma}^{-1}(y) - (x-y)}  \\
        &= \norm{u-D_\sigma(u) - (v-D_\sigma(v))}  \\
        &= \norm{\nabla g_\sigma(u) - \nabla g_\sigma(v)}  \\
        &\leq L \norm{u-v}
    \end{split}
    \end{equation}
    because $g_\sigma$ is $L$-Lipschitz. Moreover, as $J_{D_\sigma}(x) = \nabla^2 h_\sigma(x) = \id - \nabla^2 g_\sigma(x)$, 
    $\forall u \in \R^n$, $\langle \nabla^2 h_\sigma(x)u, u \rangle = \norm{u}^2 - \langle \nabla^2 g_\sigma(x)u, u \rangle \geq (1-L)\norm{u}^2$ and $h_\sigma$ is $(1-L)$-strongly convex.
    Thus continuing from last inequalities
    \begin{equation}
    \begin{split}
        \norm{\nabla \phi_\sigma(x)- \nabla \phi_\sigma(y)} &\leq \frac{L}{1-L} \norm{\nabla h_\sigma(u)- \nabla h_\sigma(v)} \\
        &= \frac{L}{1-L} \norm{D_\sigma(u)- D_\sigma(v)} \\
        &= \frac{L}{1-L} \norm{x-y} .
    \end{split}
    \end{equation}
\end{proof}

This result states that a proximal denoiser can be defined from the  denoiser~\eqref{eq:gs_denoiser}, if the gradient of the learned potential $g_\sigma$ is contractive.
In~\cite{hurault2022proximal} the Lipschitz constant of  $\nabla g_\sigma$ is softly constrained to satisfy $L_{g_\sigma}<1$, by penalizing the spectral norm $\norm{\nabla^2 g_\sigma(x+\nu_\sigma)}_S$ in the denoiser training loss. The resulting proximal denoiser has a fixed weak-convexity constant that depends only on $L_{g_\sigma}$. In the next section, we propose a way to control this value without retraining $g_\sigma$.

\subsection{More details on the regularization $\phi_\sigma$} \label{sec:more_details_reg_proxpnp}

We keep the parameterization of the potential $g_\sigma$ proposed by \cite{hurault2022proximal}:
\begin{equation} \label{eq:gsigma_proxpnp}
    g_\sigma(x) = \frac{1}{2}\norm{x - N_\sigma(x)}^2,
\end{equation}
where  $N_\sigma : \mathbb{R}^n \to \mathbb{R}^n$ is a $\mathcal{C}^\infty$ neural network with softplus activations
\begin{equation}
   s_\epsilon(x) = \frac{1}{\epsilon}\log(1 + e^{\epsilon x})
\end{equation}
As explained in \cite{hurault2021gradient}, $g_\sigma$ can be made coercive by choosing a convex compact set $C \subset \mathbb{R}^n$ where the iterates should stay and by adding an extra term
\begin{equation}
    g_{\sigma}(x) = \frac{1}{2} \norm{x - N_\sigma(x)}^2 + \frac{1}{2}\norm{x - \Pi_C(x)}^2
\end{equation}
with for example $C = [-1, 2]^{n}$ for images taking their values in $[0;1]^n$. Assuming that $L_{g_\sigma}<1$, the potential $\phi_\sigma$ obtained with Proposition~\ref{prop:proxdenoiser} then verifies: 
\begin{proposition}[Properties of $\phi_\sigma$]
With the above parameterization of $g_\sigma$, let $\phi_\sigma$ obtained from $g_\sigma$ via Proposition~\ref{prop:proxdenoiser}, then $\phi_\sigma$ verifies 
\begin{itemize}
  \item[(i)] $\phi_\sigma$ is \textit{lower-bounded}.
  \item[(ii)] $\phi_\sigma$ is $\frac{L}{L+1}$-\textit{weakly convex}.
  \item[(iii)] $\phi_\sigma$ is $\mathcal{C}^\infty$ and $\nabla \phi_\sigma$ is $\frac{L}{1-1}$-\textit{Lipschitz continuous} on $\Im(D_\sigma)$.
  \item[(iv)] $\phi_\sigma$ is \textit{coercive} when $g_\sigma$ is coercive.
  \item[(v)] \textit{$\phi_\sigma$ is real analytic on $\Im(D_\sigma)$}.
\end{itemize}
\end{proposition}
\begin{proof} 

\begin{itemize}
     \item[(i)] As $g_\sigma$ is non-negative, this follows directly from Proposition~\ref{prop:proxdenoiser}(ii).
    \item[(ii)] Proposition~\ref{prop:proxdenoiser}(i)
    \item[(iii)] Proposition~\ref{prop:proxdenoiser}(iii)
    \item[(iv)] This follows directly from $\phi_\sigma(x) \geq g_\sigma(x)$ (Proposition~\ref{prop:proxdenoiser}(ii)).
    \item[(v)] $g_\sigma$ is parameterized with the softplus activation, which is a real analytic function. By composition and sum of real analytic functions (Lemma~\ref{lem:stability-analytic}(i)) $N_\sigma$ and then subsequently $g_\sigma$ are real analytic functions. Then, by Lemma~\ref{lem:stability-analytic}(ii), $\nabla g_\sigma$ and thus $D_\sigma$ are also real analytic. By the inverse function theorem (Lemma~\ref{lem:stability-analytic}(iii)), as $\forall x \in \R^n, J_{D_\sigma}(x)>0$, $D_\sigma^{-1}$ is then real analytic on $\Im(D_\sigma)$. We finally obtain, using the expression of $\phi_\sigma$ on $\Im(D_\sigma)$ \eqref{eq:phi}, again by sum and composition, that $\phi_\sigma$ is real analytic on $\Im(D_\sigma)$.  
\end{itemize}
\end{proof}

\subsection{Relaxed Denoiser}

Once trained,  the Gradient Step Denoiser ${D_\sigma = \id - \nabla g_\sigma}$ can be relaxed as in~\cite{hurault2021gradient} with a parameter $\gamma \in [0,1]$
\begin{equation}
    \label{eq:relaxed_proximal_denoiser}
    D^\gamma_\sigma = \gamma D_\sigma + (1-\gamma)\id = \id - \gamma \nabla g_\sigma.
\end{equation}
Applying Proposition~\ref{prop:proxdenoiser} with $g^\gamma_\sigma = \gamma g_\sigma$ which has a $\gamma L_{g_\sigma}$-Lipschitz gradient, we get that if $\gamma L_g < 1$, there exists a $\frac{\gamma L_{g_\sigma}}{\gamma L_{g_\sigma}+1}$-weakly convex ${\phi^\gamma_\sigma
}$ 
such that
\begin{equation}\label{eq:relaxed_proximal_denoiser_phi}
D^\gamma_\sigma = \prox_{\phi^\gamma_\sigma},
\end{equation}
satisfying $\phi^0_\sigma=0$ and $\phi^1_\sigma=\phi_\sigma$. 
Hence, one can control the weak convexity of the regularization function by relaxing the proximal denoising operator $D^\gamma_\sigma$.

\section{PnP Proximal Gradient Descent (PnP-PGD)}\label{sec3}

In this section, we give convergence results for the  ProxPnP-PGD algorithm
\begin{equation}
\begin{split}
x_{k+1} &= D_\sigma \circ (\id - \lambda f)(x_k) \\ 
&= \prox_{\phi_{\sigma}} \circ (\id - \lambda f)(x_k)
\end{split}
\end{equation}
which is the PnP version of PGD, with plugged Proximal Denoiser \eqref{eq:proximal_denoiser}. The convergence proof proposed in~\cite{hurault2022proximal} for this algorithm is suboptimal as the weak convexity of $\phi_\sigma$ is not exploited. Doing so, we improve here the condition on the regularization parameter $\lambda$ for convergence.

\subsection{Proximal Gradient Descent  with a weakly convex function}
\label{sec:pgd}

We consider the general minimization problem 
\begin{equation}
    \label{eq:F}
    \min_x F(x):= f(x) + \phi(x) ,
\end{equation}
for a smooth nonconvex function~$f$ and a weakly convex function $\phi$ that are both bounded from below. We study under which conditions the classical Proximal Gradient Descent
\begin{equation}
    \label{eq:PGD}
    x_{k+1} \in \prox_{\tau \phi} \circ (\id - \tau \nabla f)(x_k)
\end{equation}
converges to a stationary point of~\eqref{eq:F}. 
We first show convergence of the function
values, and then convergence of the iterates, if $F$ verifies the Kurdyka-Łojasiewicz (KŁ) property (Definition~\ref{def:KŁ}). 

\begin{theorem}[Convergence of the PGD algorithm~\eqref{eq:PGD}]
\label{thm:PGD}
Assume $f$ and $\phi$ proper lsc, bounded from below with $f$ differentiable with $L_f$-Lipschitz gradient, and $\phi$ $M$-weakly convex. Then for ${\tau < \max( \frac{2}{L_f +M}, \frac{1}{L_f})}$, the iterates~\eqref{eq:PGD} verify 
 \begin{itemize}
        \item[(i)] $(F(x_k))$ is non-increasing and converges.
        \item[(ii)] The sequence has finite length \emph{i.e.} $\sum_{k=0}^{+\infty} \norm{x_{k+1}-x_k}^2 < +\infty$ and $\norm{x_{k+1}-x_k}$ converges to $0$ at rate $\min_{k < K} \norm{x_{k+1}-x_k} = \mathcal{O}(1/\sqrt{K})$.
        \item[(iii)] All cluster points of the sequence $x_k$ are stationary points of $F$.
        \item[(iv)] If the sequence $(x_k)$ is bounded and if $F$ verifies the KŁ property at the cluster points of $(x_k)$, then  $(x_k)$ converges, with finite length, to a stationary point of $F$.
    \end{itemize}
\end{theorem}
\begin{remark}
\begin{itemize}
    \item[(i)] The boundedness of the iterates $(x_k)$ is verified as soon as the objective $F$ is \textbf{coercive}. Indeed, it ensures that $\{ F(x) \leq F(x_0) \}$ is bounded and, since $F(x_k)$ is non-increasing (Theorem~\ref{thm:PGD}(i)), that the iterates remain bounded. 
    \item[(ii)] We explained in Section~\ref{sec:defs_prop} that, in practice, the KŁ property is verified as soon as $F$ is real analytic.
\end{itemize}
\end{remark}
\begin{proof} The proof follows standard arguments of the convergence analysis of the PGD in the nonconvex setting~\cite{beck2009fast,attouch2013convergence,ochs2014ipiano}.

\textbf{(i)} 
Relation~\eqref{eq:PGD} leads to
$\frac{x_{k}- x_{k+1}}{\tau} - \nabla f(x_k) \in \partial  \phi(x_{k+1})$,
by definition of the proximal operator. 
As $\phi$ is $M$-weakly convex, we have from Proposition~\ref{prop:weaklyconvex} (ii) 
\begin{equation}
\label{eq:after_weak_convexity}
\begin{split}
\phi(x_{k}) \geq &\phi(x_{k+1})+ \frac{\norm{x_{k}-x_{k+1}}^2}{\tau}-\frac{M}{2} \norm{ x_{k}-x_{k+1}}^2 +  \langle  \nabla f(x_k), x_{k+1}-x_{k} \rangle.
\end{split}
\end{equation}
The descent Lemma (Proposition \ref{prop:descent_lemma}) gives for $f$:
\begin{equation}
\begin{split}
f(x_{k+1}) \leq f(x_{k}) +  \langle  \nabla f(x_k), x_{k+1}-x_{k} \rangle + \frac{L_f}{2} \norm{ x_{k}-x_{k+1}}^2.
\end{split}
\end{equation}
Combining both inequalities, for $F=f+\phi$, we obtain
\begin{equation}
\label{eq:decrease_F}
 F(x_{k}) \geq F(x_{k+1})+ \left(\frac1\tau -\frac{M + \lambda L_f}{2} \right)    \norm{x_{k} - x_{k+1}}^2  .
\end{equation}
Therefore, if $\tau < 2/({M+\lambda L_f})
$,  $(F(x_k))$ is non-increasing. As $F$ is assumed lower-bounded, $(F(x_k))$ converges. We call $F^*$ its limit.\\
\textbf{(ii)}
            Summing~\eqref{eq:decrease_F} over $k=0,1,...,m$ gives
            \begin{equation}
                \begin{split}
                    \sum_{k=0}^m \norm{x_{k+1}-x_k}^2 \leq& \frac{1}{\frac1\tau -\frac{M + L_f}{2} }\left(F(x_0) -F(x_{m+1})\right) \\
                    \leq& \frac{1}{\frac1\tau -\frac{M + L_f}{2} } \left(F(x_0)-F^*\right) .
                \end{split}
            \end{equation}
            Therefore, $\lim_{k\to\infty} \norm{x_{k+1}-x_k} = 0$ with the convergence rate 
            \begin{equation}
            \hspace*{-1cm}\gamma_k = \min_{0 \leq i \leq k}\norm{x_{i+1}-x_i}^2 \leq \frac{1}{k}  \frac{F(x_0)-F^*}{\frac1\tau -\frac{M + L_f}{2} }.
            \end{equation}
            
\textbf{(iii)}
 Suppose that a subsequence $(x_{k_i})_i$ is converging towards $x$. Let us show that $x$ is a critical point of $F$.
By optimality of the proximal operator of $\phi$, for all $k \geq 0$
        \begin{equation} \label{eq:subdif_PGD}
                \frac{x_{k_i+1}-x_{k_i}}{\tau} - \nabla f(x_{k+1}) \in \partial \phi(x_{k_i+1}),
        \end{equation}
        where $\partial \phi$ stands for the limiting subdifferential~\eqref{eq:limiting_subdifferential}. In other words, for $\omega_{k_i} = \frac{x_{k_i}-x_{k_i-1}}{\tau} - \nabla f(x_{k_i})$, $(\omega_{k_i}, x_{k_i}) \in \operatorname{Graph} \partial \phi$. From the continuity of $\nabla f$, we have $ \nabla f (x_{k_i}) \to \nabla f (x)$. As $\norm{x_{k_i+1}-x_{k_i}}\to 0$, we get
        \begin{equation}
            \frac{x_{k_i}-x_{k_i-1}}{\tau} - \nabla f(x_{k_i}) \to - \nabla f (x).
        \end{equation}
        By closeness of $\partial \phi$ (Proposition~\ref{prop:subdif_closed}), 
        it remains to show that $\phi(x_{k_i}) \rightarrow \phi(x)$, to get $- \nabla f (x) \in \partial \phi(x)$, i.e. $x$ is a critical point of $F$ (by Proposition~\ref{prop:subdif_sum}). 
        To do so, we first use the fact that $g$ is lsc  to obtain
        \begin{equation}
            \liminf_{i \to \infty} \phi(x_{k_i})  \geq \phi(x).
        \end{equation}
        On the other hand, by optimality of the $\prox$, 
        \begin{equation} 
        \begin{split}
                    &\phi(x_{k_i+1}) + \langle x_{k_i+1}-x_{k_i}, \nabla f(x_{k_i}) \rangle \frac{1}{2 \tau}\norm{x_{k_i+1}-x_{k_i}}^2 
                    \\ &\leq \phi(x) + \langle x-x_{k_i}, \nabla f(x_{k_i}) \rangle \frac{1}{2 \tau}\norm{x-x_{k_i}}^2 .
        \end{split}
        \end{equation} 
        Using that $x_{k_i} \to x$ and $x_{k_i+1}-x_{k_i} \to 0$ when $i \to +\infty$, we get
        \begin{equation}
        \limsup_{i \to \infty} \phi(x_{k_i}) \leq \phi(x),
        \end{equation}
        and 
        \begin{equation}
        \lim_{i \to \infty} \phi(x_{k_i}) = \phi(x).
        \end{equation}

        \item[(iv)] We wish to apply Theorem 2.9 from~\cite{attouch2013convergence}. We need to verify that the sequence $(x_k)$ satisfies the three conditions H1, H2, and~H3 specified in~\cite[Section 2.3]{attouch2013convergence}:
        \begin{itemize}
        \item \textit{ H1 : Sufficient decrease condition} $\forall k \in \mathbb{N}$, 
        \begin{equation}
            f(x_{k+1}) + a \norm{x_{k+1}-x_k}^2 \leq F(x_k) .
        \end{equation}
        \item \textit{ H2 : Relative error condition}
        $\forall k \in \mathbb{N}$, there exists $\omega_{k+1} \in \partial F(x_{k+1})$ such that 
        \begin{equation}
            \norm{\omega_{k+1}} \leq b\norm{x_{k+1}-x_k} .
        \end{equation}
        \item \textit{ H3 : Continuity condition}
        There exists a subsequence $(x_{k_i})_{i\in\mathbb{N}}$ and $\hat x \in \mathbb{R}^n$ such that 
        \begin{equation}
           x_{k_i} \to \hat x \ \ \text{and} \ \ F(x_{k_i}) \to F(\hat x) \ \ \text{as} \ \ i \to +\infty .
        \end{equation}
    \end{itemize}

    Condition H1 corresponds to the sufficient decrease condition shown in (i). For condition H3, we use that, as $(x_k)$ is bounded, there exists a subsequence $(x_{k_i})$ converging towards $y$.
        Then $F(x_{k_i}) \to F(y)$ has been shown in (iii). Finally, for condition H2, from~\eqref{eq:after_weak_convexity}, we had that 
        \begin{equation}
        z_{k+1} +\lambda \nabla f(x_{k+1}) = \frac{x_{k+1} - x_k}{\tau}
        \end{equation}
        where $z_{k+1} \in \partial \phi(x_{k+1})$. Therefore, 
        \begin{equation}
       \hspace{-.5cm} \norm{z_{k+1} + \lambda\nabla f(x_{k+1})} = \frac{\norm{x_{k+1} - x_k}}{\tau}
        \end{equation} which gives condition H2. 
\end{proof}

\subsection{ProxPnP Proximal Gradient Descent (ProxPnP-PGD)}\label{sec:pgd_prox}

Equipped with the convergence of PGD, we can now study the convergence of \emph{ProxPnP-PGD}, the PnP-PGD algorithm with plugged Proximal Denoiser~\eqref{eq:proximal_denoiser}:
\begin{equation}
    \label{eq:ProxPnP-PGD}
     x_{k+1} = D_\sigma \circ (\id - \lambda f)(x_k) 
     = \prox_{\phi_{\sigma}} \circ (\id - \lambda f)(x_k).
\end{equation}
This algorithm corresponds to the PGD algorithm~\eqref{eq:PGD} with fixed stepsize~$\tau=1$ and targets stationary points of the functional $F$ defined by:
\begin{equation}
    \label{eq:F_tau}
    F := \reglambda  f + \phi_\sigma.
\end{equation}

At each iteration, $x_k \in \Im(D_\sigma)$, and thus $\phi_\sigma$ defined in~\eqref{eq:phi} is tractable along the algorithm. The value of the objective function~\eqref{eq:F_tau} at $x_k$ is 
\begin{equation}
    F(x_k) = \reglambda  f(x_k) +  g_\sigma(z_k) -\frac{1}{2}\norm{z_k-x_k}^2 + K
\end{equation}
where ${z_k = (\id - \lambda \nabla f)(x_{k-1})=D_\sigma^{-1}(x_k)}$.

The following result, obtained from Theorem~\ref{thm:PGD},  improves~\cite{hurault2022proximal} using the fact that the potential $\phi_\sigma$ is not any nonconvex function but a weakly convex one. 
\begin{corollary}[Convergence of ProxPnP-PGD]
    \label{cor:PnP-PGD}
    Let $f : \mathbb{R}^n \to \mathbb{R} \cup \{+\infty\}$ be differentiable with $L_f$-Lipschitz gradient, bounded from below. Assume that $L_{g_\sigma}$ the Lipschitz constant of $\nabla g_\sigma$ verifies ${L_{g_\sigma}<1}$. Then, for ${\reglambda L_f < \frac{L_{g_\sigma}+2}{L_{g_\sigma}+1}}$, the iterates $(x_k)$ given by the iterative scheme~\eqref{eq:ProxPnP-PGD} verify
    \begin{itemize}
        \item[(i)] $(F(x_k))$ is non-increasing and converges.
        \item[(ii)] The sequence has finite length \emph{i.e.} $\sum_{k=0}^{+\infty} \norm{x_{k+1}-x_k}^2 < +\infty$ and $\norm{x_{k+1}-x_k}$ converges to $0$ at rate $\min_{k < K} \norm{x_{k+1}-x_k} = \mathcal{O}(1/\sqrt{K})$.
        \item[(iii)] If $f$ is real analytic, and $g_\sigma$ is coercive, then the iterates $(x_k)$ converge towards a critical point of $F$. 
    \end{itemize}
\end{corollary}
\begin{proof}
    This is a direct application of Theorem~\ref{thm:PGD} with $\lambda f$ the smooth function and with $\phi_\sigma$ the weakly convex function. The stepsize condition 
    \begin{equation}
        \tau < \max( \frac{2}{\lambda L_f + \frac{L}{L+1}}, \frac{1}{\lambda L_f})
    \end{equation}
    becomes, with $\tau = 1$, 
    \begin{align}
    &\Leftrightarrow \lambda L_f + \frac{L_{g_\sigma}}{L_{g_\sigma}+1} < 2 \text{ or } \lambda L_f < 1 \\
    &\Leftrightarrow \lambda L_f < \frac{L_{g_\sigma}+2}{L_{g_\sigma}+1}.
    \end{align}
    For (iii), as shown in Section~\ref{sec:more_details_reg_proxpnp}, $\phi_\sigma$ is real analytic on $\Im(D_\sigma)$, lower-bounded and coercive if $g_\sigma$ is coercive. As $f$ is assumed real analytic, by Lemma~\ref{lem:stability-analytic}, $F = \lambda f + \phi_\sigma$ is also real analytic and thus KŁ on $\Im(D_\sigma)$. 
    This is enough for the proof of Theorem~\ref{thm:PGD} to hold because any limit point of $(x_k)$ is a fixed point of the PnP-PGD operator ${D_\sigma \circ (\id - \lambda f)}$ and thus belongs to $\Im(D_\sigma)$. 
\end{proof}

By exploiting the weak convexity of $\phi_\sigma$, the convergence condition $\lambda L_f<1$  of~\cite{hurault2022proximal} is here replaced by  $\reglambda L_f < \frac{L_{g_\sigma}+2}{L_{g_\sigma}+1}$. Even if the bound is improved, we are still limited to regularization parameters satisfying $\lambda L_f<2$. 

\medbreak 

In the same way, it was proven in \cite[Theorem 4.3]{hurault2022proximal} that the PnP version of Douglas-Rachford Splitting (DRS) scheme with Proximal denoiser (ProxPnP-DRS), which writes 
\begin{subequations}\label{eq:ProxPnP-DRS}
\begin{empheq}[left=\empheqlbrace]{align}
    y_{k+1} &\in \prox_{\lambda f}(x_k) \label{eq:ProxPnP-DRS_y}\\ 
     z_{k+1} &= \prox_{\phi_\sigma}(2y_{k+1}-x_{k}) \label{eq:ProxPnP-DRS_z}\\ 
     x_{k+1} &=  x_{k} + 2\beta (z_{k+1}-y_{k+1}) , \label{eq:ProxPnP-DRS_x}
\end{empheq}
\end{subequations}
converges towards a critical point of $\lambda f + \phi_\sigma$ provided $\lambda L_f<1$ (see Section~\ref{sec:DRS} for more details). 

\medbreak 

In both cases, the value of the regularization trade-off parameter $\lambda$ is then limited. This is an issue when restoring an image with mild degradations for which relevant solutions are obtained with a low amount of regularization and a dominant data-fidelity term in $F = \lambda f + g$. We could argue that this is a not a problem as the regularization strength is also regulated by the $\sigma$ parameter, which we are free to tune manually. However, it is observed in practice, for instance in \cite{hurault2021gradient}, that the performance of PnP method greatly benefits from the ability to tune the two regularization parameters. 

Given this limitation, our objective is to design new convergent PnP algorithms with proximal denoiser, and with minimal restriction on the regularization parameter~$\lambda$. We not only wish to adapt the denoiser but also the original optimization schemes of interest. 

\section{PnP Relaxed Proximal Gradient Descent (PnP-$\alpha$PGD)}\label{sec4}

In this section, we study a relaxation of the general Proximal Gradient Descent algorithm, called \emph{$\alpha$PGD}, such that when used with the proximal denoiser, the corresponding PnP scheme \emph{ProxPnP-$\alpha$PGD} can converge for a wider range of regularization parameters~$\lambda$.  

\subsection{$\alpha$PGD algorithm}

We first introduce the $\alpha$PGD algorithm for solving the general problem
\begin{equation}
    \label{eq:pb_foralphaPGD}
    \min_{x \in \mathbb{R}^n} f(x) + \phi(x)
\end{equation}
for $f,\phi : \R^n \to \R \cup \{ +\infty \}$ with $f$ convex and smooth and $\phi$ weakly convex. The algorithm writes, for ${\alpha \in (0,1)}$, 
\begin{subequations} \label{eq:alphaPGD}
\begin{empheq}[left=\empheqlbrace]{align}
q_{k+1} &= (1-\alpha) y_k + \alpha x_k   \label{eq:PGD2_q} \\
x_{k+1} &= \prox_{\tau \phi}(x_k -\tau \nabla f(q_{k+1}))  \label{eq:PGD2_x}\\
y_{k+1} &= (1-\alpha) y_k + \alpha  x_{k+1}. \label{eq:PGD2_y}
\end{empheq}
\end{subequations}  
Algorithm~\eqref{eq:alphaPGD} with $\alpha=1$ exactly corresponds to the PGD algorithm~\eqref{eq:PGD}.
This scheme is reminiscent of~\citep{tseng2008accelerated} (taking $\alpha=\theta_k$ and $\tau=\frac{1}{\theta_k L_f}$ in Algorithm~1 of~\citep{tseng2008accelerated}), which generalizes Nesterov-like accelerated proximal gradient methods \citep{beck2009fast,nesterov2013gradient}. Inspired by~\cite{lan2018optimal}, this scheme was derived from the Primal-Dual algorithm~\citep{CP11} with Bregman proximal operator~\citep{chambolle2016ergodic}. We describe this derivation in Appendix~\ref{app:derivation_from_PD}. 

\paragraph{$\alpha$PGD convergence} In the convex setting, using the convergence result from \citep{chambolle2016ergodic} of the Bregman Primal-Dual algorithm (from which $\alpha$PGD was derived), one can prove convergence of $\alpha$PGD for $\tau \lambda L_f < 2$ and small enough $\alpha$.  
However, in our case, the objective $F = \lambda f + \phi$ is nonconvex and a new convergence result needs to be derived. This is done with the following theorem.



\begin{theorem}[Convergence of $\alpha$PGD~\eqref{eq:alphaPGD}]\label{thm:alphaPGD}
Assume $f$ and $\phi$ proper, lsc, lower-bounded, with $f$ convex and $L_f$-smooth and $\phi$ $M$-weakly convex. Then for $\alpha \in (0,1)$ and
$\tau< \min\left(\frac{1}{\alpha L_f},\frac{\alpha}{M}\right)$,\,
the updates~\eqref{eq:alphaPGD} verify
\begin{itemize}
    \item[(i)] $F(y_k) + \frac{\alpha}{2\tau}\left(1-\frac1\alpha\right)^2 \norm{y_k - y_{k-1}}^2$ is non-increasing and converges.
    \item[(ii)]the sequence $y_k$ has finite length \emph{i.e.} $\sum_{k=0}^{+\infty} \norm{y_{k+1}-y_k}^2 < +\infty$ and $ \norm{y_{k+1}-y_k}$ converges to $0$ at rate $\min_{k < K} \norm{y_{k+1}-y_k} = \mathcal{O}(1/\sqrt{K})$.
    \item[(iii)] All cluster points of the sequence $y_k$ are stationary points of $F$. In particular, if $g$ is coercive, then $y_k$ has a subsequence converging towards a stationary point of $F$.
\end{itemize}
\end{theorem}
\begin{remark} \label{rem:point_conv}
   With this theorem, $\alpha$PGD is shown to verify convergence of the iterates and of the norm of the residual to $0$. Note that we do not have here the analog of Theorem~\ref{thm:PGD}(iv) on the iterates convergence using the KŁ hypothesis. Indeed, as we detail at the end of the proof, the nonconvex convergence analysis with KŁ functions from~\cite{attouch2013convergence} or~\cite{ochs2014ipiano} does not extend to our case. 
    \end{remark}
    \begin{remark} As shown in Appendix~\ref{app:better_bound}, a better bound on $\tau$ could be found, but with little numerical gain. Moreover, when $\alpha=1$, algorithms $\alpha$PGD~\eqref{eq:alphaPGD} and PGD~\eqref{eq:PGD} are equivalent, but we get a slightly worst bound in Theorem~\ref{thm:alphaPGD} than in Theorem~\ref{thm:PGD} ($ \tau < \min \left( \frac{1}{ L_f}, \frac{1}{M}\right)\leq \frac2{ L_f+M}$). Nevertheless, when used with $\alpha<1$, we next show that the relaxed algorithm is more relevant in the perspective of PnP with proximal denoiser.
    \end{remark}
    \begin{remark}
        The data-fidelity term $f$ is here assumed convex with Lipschitz gradient. This is verified by the classical $L^2$ data-fidelity term that we use in our experiments of Section~\ref{sec:expes}. However, this excludes several data-fidelity terms such as the $L^1$ or the Kullback-Leiber distances. 
    \end{remark}
    
\begin{proof}

The proof relies on Lemma~\ref{prop:descent_lemma} and Proposition~\ref{prop:weaklyconvex}. It follows the general strategy of the proofs in~\cite{tseng2008accelerated},
and also requires the convexity of $f$. 

\noindent {\bf (i) and (ii) :} 
We can write~\eqref{eq:PGD2_x} as 
\begin{equation}
\begin{split}
      x_{k+1} &\in \argmin_{y \in \mathbb{R}^n} \phi(y) + \langle \nabla f(q_{k+1}),y-x_k \rangle + \frac{1}{2\tau} \norm{y-x_k}^2 \\
      &\in \argmin_{y \in \mathbb{R}^n} \phi(y)+ f(q_{k+1}) + \langle \nabla f(q_{k+1}),y-q_{k+1} \rangle + \frac{1}{2\tau} \norm{y-x_k}^2 \\
     &\in \argmin_{y \in \mathbb{R}^n} \Phi(y) +\frac{1}{2\tau} \norm{y-x_k}^2
\end{split}
\end{equation}
with $\Phi(y) := \phi(y) +  f(q_{k+1}) + \langle \nabla f(q_{k+1}),y-q_{k+1}\rangle$.  
As $\phi$ is $M$-weakly convex, so is~$\Phi$.
The three-points inequality of Proposition~\ref{prop:weaklyconvex} (iii) applied to $\Phi$ thus gives  $\forall y \in\mathbb{R}^n$,
\begin{equation}
\begin{split}
\Phi(y)    +    \frac{1}{2\tau}  \norm{y-x_k}^2  \geq \Phi(x_{k+1})  +  \frac{1}{2\tau}\norm{ x_{k+1}-x_k}^2 +  \left(\frac{1}{2\tau}  -  \frac{M}{2} \right)  \norm{x_{k+1}-y}^2
\end{split}
\end{equation}
that is to say,
\begin{equation}
\begin{split}
\label{eq:after_3point}
 & \phi(y) +  f(q_{k+1}) +  \langle \nabla f(q_{k+1}),y-q_{k+1} \rangle +\frac{1}{2\tau} \norm{y-x_k}^2 \geq \\
 &\phi(x_{k+1}) +  f(q_{k+1}) +  \langle \nabla f(q_{k+1}),x_{k+1}-q_{k+1} \rangle + \frac{1}{2\tau}\norm{x_{k+1}-x_k}^2  + \left(\frac{1}{2\tau} - \frac{M}{2} \right) \norm{x_{k+1}-y}^2.
 \end{split}
\end{equation}
Using relation~\eqref{eq:PGD2_y},  the descent Lemma (Proposition~\ref{prop:descent_lemma}) as well as the convexity on~$f$, we obtain
\begin{align} 
        \nonumber& f(q_{k+1}) + \langle \nabla f(q_{k+1}),x_{k+1}-q_{k+1} \rangle \\\nonumber
        =& f(q_{k+1}) +  \left\langle \nabla f(q_{k+1}), \frac{1}{\alpha} y_{k+1} + \left(1-\frac{1}{\alpha}\right) y_{k} - q_{k+1} \right\rangle \\
        \nonumber=&\frac{1}{\alpha} \Big( f(q_{k+1}) + \langle \nabla f(q_{k+1}),y_{k+1} - q_{k+1} \rangle \Big) +
        \left(1-\frac{1}{\alpha}\right)\Big( f(q_{k+1}) + \langle \nabla f(q_{k+1}),y_{k}- q_{k+1} \rangle \Big) \\\nonumber
        \geq &\frac{1}{\alpha} \left(  f(y_{k+1}) - \frac{L_f}{2} \norm{y_{k+1}-q_{k+1}}^2 \right) +
        \left(1-\frac{1}{\alpha}\right) \Big( f(q_{k+1}) + \langle \nabla f(q_{k+1}),y_{k}- q_{k+1} \rangle \Big) \\
        \geq &\frac{1}{\alpha} \left(  f(y_{k+1}) - \frac{L_f}{2}\norm{y_{k+1}-q_{k+1}}^2 \right) +        \left(1-\frac{1}{\alpha}\right)f(y_k).  \label{relation2}
\end{align}
From relations~\eqref{eq:PGD2_q} and~\eqref{eq:PGD2_y}, we have $y_{k+1}-q_{k+1} = \alpha(x_{k+1}-x_{k})$. Combining this expression with~\eqref{eq:after_3point} and~\eqref{relation2}, we have for all $y \in\mathbb{R}^n$
    \begin{align}
     & \phi(y) +  \left(\frac{1 }{\alpha}-1\right) f(y_k) + \frac{1}{2\tau} \norm{y-x_k}^2   f(q_{k+1})+  \langle \nabla f(q_{k+1}),y-q_{k+1} \rangle \\
     \geq &\phi(x_{k+1})  +   \frac{1}{\alpha} f(y_{k+1})  +   \left( \frac{1}{2\tau} -  \frac{M}{ 2 }\right)   \norm{x_{k+1}-y}^2  +\left(\frac{1}{2\tau}  -  \frac{\alpha L_f}{2}\right) \norm{x_{k+1}-x_k}^2 .\nonumber
     \end{align}
Using the convexity of $f$ we get for all $y \in\mathbb{R}^n$,
\begin{align}\label{eq:relation1}
     & \phi(y)  +     f(y)  +  \left(\frac{1}{\alpha}-1   \right)    f(y_k) + \frac{1}{2\tau} \norm{y-x_k}^2 
      \\
     \geq &\phi(x_{k+1})    +    \frac{1 }{\alpha}  f(y_{k+1})    +   \left(  \frac{1}{2\tau}  -  \frac{\alpha  L_f}{2} \right)  \norm{x_{k+1} - x_k}^2  +  \left(\frac{1}{2\tau}- \frac{M}{ 2 }\right) \norm{x_{k+1}-y}^2, 
   \nonumber
    \end{align}
while weak convexity of $\phi$ with relation~\eqref{eq:PGD2_y} gives 
\begin{align}
\label{eq:phiconvex}
\phi(x_{k+1}) \geq &\,\frac{1}{\alpha} \phi(y_{k+1}) + \left(1-\frac{1}{\alpha}\right) \phi(y_k) - \frac{M}{2}(1-\alpha) \norm{y_k - x_{k+1}}^2.
\end{align}
Using~\eqref{eq:relation1},~\eqref{eq:phiconvex}  and $F =  f + \phi$, we obtain $\forall y \in\mathbb{R}^n$
\begin{align}\label{eq:descent_semiconvex}
    & \left(\frac{1}{\alpha}-1\right) (F(y_k) - F(y)) + \frac{1}{2\tau} \norm{y-x_k}^2 \geq
     \frac{1}{\alpha} (F(y_{k+1})   -  F(y))   \\ &+   \left(\frac{1}{2\tau} - \frac{\alpha  L_f}{2}\right)   \norm{x_{k+1} -x_k}^2 +    \left(  \frac{1}{2\tau} -   \frac{M}{ 2 } \right)    \norm{x_{k+1} -y}^2 
       -  \frac{M}{2} (1 -\alpha)   \norm{y_k  - x_{k+1}}^2.
\nonumber
\end{align}
For $y=y_k$, we get 
\begin{equation}
\begin{split}
 \frac{1}{\alpha} (F(y_{k}) - F(y_{k+1})) \geq &
\left(\frac{1}{2\tau} - \frac{\alpha L_f}{2}\right) \norm{x_{k+1}-x_k}^2 \\ &- \frac{1}{2\tau} \norm{y_k-x_k}^2  \left(\frac{1}{2\tau} -  \frac{M(2-\alpha)}{ 2 }\right) \norm{x_{k+1}-y_k}^2.
\end{split}
\end{equation}
For $\alpha \in (0,1)$, using that 
\begin{align}
    & y_k - x_{k+1} = \frac{1}{\alpha}(y_k - y_{k+1}) \\
    & y_{k} - x_{k} = \left(1-\frac{1}{\alpha}\right)(y_k - y_{k-1})
\end{align}
we get 
\begin{equation}
\label{eq:before_cases}
  \begin{split}
   F(y_{k}) - F(y_{k+1}) \geq& - \frac{\alpha}{2\tau}\left(1-\frac1\alpha\right)^2 \norm{y_k-y_{k-1}}^2 \\ & 
     + \alpha \left(\frac{1}{2\tau} - \frac{\alpha  L_f}{2}\right) \norm{x_{k+1}-x_k}^2   \\ &+ \frac1\alpha \left(\frac{1}{2\tau}- \frac{M(2-\alpha)}{ 2 }\right) \norm{y_{k+1}-y_k}^2.
     \end{split}
\end{equation}
With the assumption $\tau \alpha  L_f < 1$, the second term of the right-hand side is non-negative and therefore,
\begin{equation} \label{eq:descent_aPGD}
  \begin{split}
  F(y_{k}) - F(y_{k+1}) \geq &- \frac{\alpha}{2\tau}\left(1-\frac1\alpha\right)^2 \norm{y_k-y_{k-1}}^2 \\ &+ \frac1\alpha \left(\frac{1}{2\tau}- \frac{M(2-\alpha)}{ 2 } \right) \norm{y_{k+1}-y_k}^2. \\
     = &-\delta \norm{y_k - y_{k-1}}^2 + \delta \norm{y_{k+1} - y_{k}}^2 + (\gamma-\delta) \norm{y_{k+1} - y_{k}}^2
    \end{split}
    \end{equation}
with 
\begin{align}
    \delta = \frac{\alpha}{2\tau}\left(1-\frac1\alpha\right)^2 \text{ and }
    \gamma =  \frac1\alpha \left(\frac{1}{2\tau}- \frac{M(2-\alpha)}{ 2 }\right).
\end{align}

We now make use of the following lemma. 
\begin{lemma}[\cite{bauschke2011convex}]
\label{lem:an_bn}
Let $(a_n)_{n \in \mathbb{N}}$ and $(b_n)_{n \in \mathbb{N}}$ be two real sequences such that $b_n \geq 0$ $\forall n \in \mathbb{N}$, $(a_n)$ is bounded from below and $a_{n+1}+b_n \leq a_n$ $\forall n \in \mathbb{N}$. Then $(a_n)_{n \in \mathbb{N}}$ is a monotonically non-increasing and convergent sequence and $\sum_{n \in \mathbb{N}} b_n < +\infty$.
\end{lemma}
To apply Lemma~\ref{lem:an_bn} we look for a stepsize satisfying 
\begin{equation}
    \gamma - \delta > 0 \quad \text{i.e.} \quad \tau < \frac{\alpha}{M}.
\end{equation}
Therefore, hypothesis $\tau < \min \left(\frac{1}{\alpha  L_f},\frac{\alpha}{M}\right)$ gives that $(F(y_k) + \delta \norm{y_k - y_{k-1}}^2)$  is a non-increasing and convergent sequence and that $\sum_k \norm{y_k - y_{k+1}}^2 < + \infty$.

\noindent
{\bf (iii)} The  proof of this result is an extension of the proof of Theorem~\ref{thm:PGD}(ii) in the context of the classical PGD. 
Suppose that a subsequence $(y_{k_i})$ is converging towards $y$. Let us show that $y$ is a critical point of $F$.
From~\eqref{eq:PGD2_x}, we have 
\begin{equation}
        \frac{x_{k+1}-x_k}{\tau} - \nabla f(q_{k+1}) \in \partial \phi (x_{k+1}).
\end{equation}
First we show that $x_{k_i+1}-x_{k_i} \to 0$. We have $\forall k > 1 $,
\begin{equation}
\begin{split}
\norm{x_{k+1}-x_k} &= \norm{ \frac{1}{\alpha} y_{k+1} + (1-\frac{1}{\alpha})y_{k} - \frac{1}{\alpha}y_{k} - (1-\frac{1}{\alpha})y_{k-1} } \\
&\leq  \frac{1}{\alpha}\norm{y_{k+1}-y_k} +  (\frac{1}{\alpha}-1)\norm{y_{k}-y_{k-1}} \\
&\to 0.
\end{split}
\end{equation}
From $\eqref{eq:PGD2_q}$, we also get $\norm{q_{k+1}-q_k} \to 0$. Now, let us show that $x_{k_i} \to y$ and $q_{k_i} \to y$. First using $\eqref{eq:PGD2_y}$, we have
\begin{equation}
\begin{split}
\norm{x_{k_i}-y} &\leq \norm{x_{k_i+1}-y}  + \norm{x_{k_i+1}-x_{k_i}} \\
&\leq \frac{1}{\alpha}\norm{y_{k_i+1}-y} + (\frac{1}{\alpha}-1)\norm{y_{k_i}-y}  + \norm{x_{k_i+1}-x_{k_i}} \\
&\to 0.
\end{split}
\end{equation}
Second, from $\eqref{eq:PGD2_q}$, we get in the same way ${q_{k_i} \to y}$. From the continuity of $\nabla f$, we obtain ${\nabla f (q_{k_i}) \to \nabla f(y)}$ and therefore 
\begin{equation}
     \frac{x_{k_i}-x_{k_i-1}}{\tau} -  \nabla f(q_{k_i}) \to -  \nabla f (y).
\end{equation}
If we can also show that $\phi(x_{k_i}) \rightarrow \phi(y)$, we get from the closeness of the limiting subdifferential (Proposition~\ref{prop:subdif_closed}) that $ -  \nabla f (y) \in \partial \phi(y)$ \emph{i.e.} $y$ is a critical point of $F$. 

Using the fact that $\phi$ is lsc and $x_{k_i} \to y$, we have 
\begin{equation}
    \liminf_{i \to \infty} \phi(x_{k_i})  \geq \phi(y).
\end{equation}
On the other hand, with Equation \eqref{eq:relation1} for $k+1 = k_i$, taking $i \to + \infty$, $\norm{y-x_{k_i+1}} \to 0$,  $\norm{y-x_{k_i}} \to 0$, $f(y_{k_i}) \to f(y)$,  $f(y_{k_i+1}) \to f(y)$ and we get 
\begin{equation}
\limsup_{i \to \infty} \phi(x_{k_i}) \leq \phi(y),
\end{equation}
and therefore
\begin{equation}
\lim_{i \to \infty} \phi(x_{k_i}) = \phi(y).
\end{equation}
As $f$ is lower-bounded, if $\phi$ is coercive, so is $F$ and by (i) the iterates $(y_k)$ remain bounded and $(y_k)$ admits a converging subsequence. 

\paragraph{Remark on the convergence of the iterates with the KL hypothesis}
\label{app:precision_KL}

In order to prove a result similar to Theorem~\ref{thm:PGD} (iv) on the convergence of the iterates with the KL hypothesis, we cannot directly apply Theorem 2.9 from~\cite{attouch2013convergence} on $F$ as the objective function $F(x_k)$ by itself does not decrease along the sequence but $F(x_k) + \delta \norm{x_{k+1}-x_k}^2$ does (where $\delta = \frac{\alpha}{2\tau}\left(1-\frac1\alpha\right)^2$). 

Our situation is more similar to the variant of this result presented in~\cite[Theorem 3.7]{ochs2014ipiano}. 
Indeed, denoting $\mathcal{F} : \mathbb{R}^n \times \mathbb{R}^n \to  \mathbb{R}$ defined as $\mathcal{F}(x,y) = F(x) + \delta \norm{x-y}^2$ and considering $\forall k \geq 1$, the sequence $z_k = (y_k,y_{k-1})$ with $y_k$ following our algorithm, we can easily show that $z_k$ verifies the conditions H1 and H3 specified in~\cite[Section 3.2]{ochs2014ipiano}. However, condition H2 does not extend to our algorithm. 
\end{proof}

\subsection{ProxPnP-$\alpha$PGD algorithm}

The $\alpha$PGD algorithm~\eqref{eq:alphaPGD} gives birth to the PnP-$\alpha$PGD algorithm by replacing the proximity operator $\prox_{\tau g}$ by a Gaussian denoiser $D_\sigma$. Now, similar to what was done in Section~\ref{sec:pgd_prox} with the PGD algorithm, in this section, we study the convergence of this PnP-$\alpha$PGD with our particular Proximal Gradient-Step Denoiser \eqref{eq:proximal_denoiser} ${D_\sigma = \prox_{\phi_\sigma}}$ and stepsize $\tau=1$. 
This corresponds to the following algorithm, which we refer to ProxPnP-$\alpha$PGD.

\begin{subequations} \label{eq:ProxPnP-alphaPGD} 
\begin{empheq}[left=\empheqlbrace]{align}
q_{k+1} &= (1-\alpha) y_k + \alpha x_k   \label{eq:PnP-PGD2_q} \\
x_{k+1} &= D_\sigma(x_k  - \lambda\nabla f(q_{k+1}))  \label{eq:PnP-PGD2_x}\\
y_{k+1} &= (1-\alpha) y_k + \alpha  x_{k+1}  \label{eq:PnP-PGD2_y}
\end{empheq}
\end{subequations} 

Like ProxPnP-PGD \eqref{eq:ProxPnP-PGD}, the ProxPnP-$\alpha$PGD scheme targets the critical points of the explicit functional $F = \lambda f + \phi_\sigma$ where $\phi_\sigma$ is obtained from the deep potential $g_\sigma$ via Proposition~\ref{prop:proxdenoiser}. Applying the previous Theorem~\ref{thm:alphaPGD} on the convergence of the $\alpha$PGD algorithm with $\tau=1$ and $g = \phi_\sigma$ which is $M = \frac{L_{g_\sigma}}{L_{g_\sigma}+1}$ weakly convex, we get the following convergence result for ProxPnP-$\alpha$PGD.

\begin{corollary}[Convergence of ProxPnP-$\alpha$PGD~\eqref{eq:ProxPnP-alphaPGD}]
    \label{cor:ProxPnP-alphaPGD} 
    Let $f : \mathbb{R}^n \to \mathbb{R} \cup \{+\infty\}$ convex and differentiable with $L_f$-Lipschitz gradient, bounded from below. Assume that $L_{g_\sigma}$ the Lipschitz constant of $\nabla g_\sigma$ verifies $L_{g_\sigma}<1$. Let $M = \frac{L_{g_\sigma}}{L_{g_\sigma}+1}$ be the weak convexity constant of $\phi_\sigma$ obtained from $g_\sigma$ via Proposition~\ref{prop:proxdenoiser}. Then, for any $\alpha \in (0,1)$ such that 
    \begin{equation}
    \label{eq:PnP-PGD2-conditions}\small 
    \lambda L_f M  < 1  \text{  and  }  M<\alpha < 1/(\lambda L_f)   
    \end{equation}
    the iterates given by the iterative scheme~\eqref{eq:ProxPnP-alphaPGD} verify
    \begin{itemize}
        \item[(i)] $F(y_k) + \frac{\alpha}{2}\left(1-\frac1\alpha\right)^2 \norm{y_k - y_{k-1}}^2$ is non-increasing and converges.
        \item[(ii)] The sequence $y_{k}$ as finite length \emph{i.e.} $\sum_{k=0}^{+\infty} \norm{y_{k+1}-y_k}^2 < +\infty$ and $\norm{y_{k+1}-y_k}$ converges to $0$ at rate $\min_{k < K} \norm{y_{k+1}-y_k} = \mathcal{O}(1/\sqrt{K})$. 
        \item[(iii)] If $g_\sigma$ is coercive, then $(y_k)$ has a subsequence converging towards a critical point of~$F$. 
    \end{itemize}
\end{corollary}
The existence of $\alpha\in(0,1)$ satisfying relation~\eqref{eq:PnP-PGD2-conditions} is ensured as soon as ${\lambda L_f M < 1}$. As a consequence, when $M$ gets small (\emph{i.e} $\phi_\sigma$ gets "more convex") $\lambda L_f$ can get arbitrarily large. This is a major advance compared to ProxPnP-PGD, which was limited to $\lambda L_f < 2$ in Corollary~\ref{cor:PnP-PGD}, even for convex $\phi$ ($M=0$).
To further exploit this property, we now consider the relaxed denoiser $D^\gamma_\sigma$~\eqref{eq:relaxed_proximal_denoiser} that is associated to a function $\phi_\sigma^\gamma$ with a tunable weak convexity constant $M^\gamma$.

\begin{corollary}[Convergence of ProxPnP-$\alpha$PGD with relaxed denoiser]\label{thm:PnP-PGD3}
Let ${F := \reglambda  f + \phi^\gamma_\sigma}$, with the $M^\gamma=\frac{\gamma L_{g_\sigma}}{\gamma L_{g_\sigma}+1}$-weakly convex potential $\phi^\gamma_\sigma$ introduced in~\eqref{eq:relaxed_proximal_denoiser_phi} and $L_{g_\sigma}<1$. Then, for $M^\gamma <\alpha < 1/(\lambda L_f)$, the iterates $x_k$ given by the ProxPnP-$\alpha$PGD~\eqref{eq:ProxPnP-alphaPGD} with $\gamma$-relaxed denoiser $D^\gamma_\sigma$ defined in~\eqref{eq:relaxed_proximal_denoiser} verify the convergence properties (i)-(iii) of Theorem~\ref{thm:alphaPGD}.
\end{corollary}
Therefore, using the $\gamma$-relaxed denoiser $D^\gamma_\sigma = \gamma D_\sigma + (1-\gamma)\id$, the overall convergence condition on $\lambda$ is now 
\begin{equation}
    \lambda < \frac{1}{L_f} \left( 1 + \frac{1}{\gamma L} \right) .
\end{equation}
Provided $\gamma$ gets small, $\lambda$ can be arbitrarily large.  Small $\gamma$ means small amount of regularization brought by denoising at each step of the PnP algorithm. Moreover, for small $\gamma$, the targeted regularization function $\phi_\sigma^\gamma$ gets close to a convex function. It has already been observed that deep convex regularization can be suboptimal compared to more flexible nonconvex ones \citep{cohen2021has}. Depending on the inverse problem, and on the necessary amount of regularization,  the choice of the couple $(\gamma,\lambda)$ will be of paramount importance for efficient restoration.  

\section{PnP Douglas-Rachford Splitting (PnP-DRS)} \label{sec:DRS}

We now focus on the ProxPnP version of the Douglas-Rachford Splitting (DRS). It corresponds to a classical DRS algorithm, with fixed stepsize $\tau = 1$, for optimizing the function $\lambda f + \phi_\sigma$. The algorithm writes, for $\beta \in (0,1]$,
\begin{subequations}\label{eq:ProxPnP-DRS2}
\begin{empheq}[left=\empheqlbrace]{align}
        y_{k+1} &= \prox_{\phi_\sigma}(x_k) \label{eq:DRS2_y} \\ 
        z_{k+1} &= \prox_{\lambda f}(2y_{k+1}-x_{k}) \label{eq:DRS2_z}
        \\ 
        x_{k+1} &=  x_{k} + 2\beta (z_{k+1}-y_{k+1}) \label{eq:DRS2_x}
\end{empheq}
\end{subequations}
or equivalently
\begin{equation}
    x_{k+1} \in \left( \beta \rprox_{\lambda f} \circ \rprox_{ \phi_\sigma} \circ + (1-\beta)\id \right) (x_k) .
\end{equation}
The case $\beta=\frac12$ corresponds to the standard version of Douglas-Rachford Splitting, and is equivalent to ADMM, while $\beta = 1$ is generally referred to as Peaceman-Rachford splitting. 

The convergence of this algorithm in the nonconvex setting, studied in \cite{li2016douglas} and~\cite{themelis2020douglas}, requires one of the two functions to be differentiable with globally Lipschitz gradient. Assuming that the data-fidelity term $f$ satisfies this property, we get convergence of the DRS algorithm provided that $\lambda L_f < 1$ \citep[Theorem 4.3]{hurault2022proximal}. In the experiments, this version of DRS will be referred to \emph{DRSdiff}. 

Convergence of this algorithm without the restriction on the parameter $\lambda$ is possible if $\phi_\sigma$ has a globally Lipschitz gradient. However, $\phi_\sigma$ obtained from Proposition~\ref{prop:proxdenoiser} is differentiable only on $\Im(D_\sigma) = \Im(\prox_{\phi_\sigma})$. In \cite{hurault2022proximal}, convergence of the iterates was proved assuming $\Im(D_\sigma)$ to be convex in order to keep the descent lemma for $\phi_\sigma$ valid on $\Im(D_\sigma)$ (Proposition~\ref{prop:descent_lemma}). However, the convexity of $\Im(D_\sigma)$ is difficult to verify and unlikely to be true in practice. 

Instead, we propose a new adaptation of the original convergence results from~\cite{themelis2020douglas,li2016douglas} on the convergence of the DRS algorithm with nonconvex functions. We directly write the convergence result in the ProxPnP case of interest. In particular, we consider a very general data-fidelity term $f$ \textbf{nonconvex and non-differentiable} and $\phi_\sigma$ obtained from Proposition~\ref{prop:proxdenoiser} via the proximal Gradient step denoiser ${D_\sigma = \id - \nabla g_\sigma = \prox_{\phi_\sigma}}$. As detailed in Section~\ref{sec:more_details_reg_proxpnp}, $\phi_\sigma$ is $\frac{L_{g_\sigma}}{1+L_{g_\sigma}}$-weakly convex on $\R^n$ and has $\frac{L_{g_\sigma}}{1-L_{g_\sigma}}$ Lipschitz gradient on $\Im(D_\sigma)$. Our convergence result makes use of these two properties.
The main adaptation concerns the proof of the sufficient decrease of the Douglas-Rachford envelope
\begin{equation}\label{eq:DRE}
\begin{split}
    F^{DR}(x,y,z) &= \reglambda  f(z) + \phi_\sigma(y) + \frac{1}{\tau}\langle y-x, y-z \rangle + \frac{1}{2}\norm{y-z}^2.
\end{split}
\end{equation}

\begin{theorem}[Convergence of ProxPnP-DRS] \label{thm:ProxPnP-DRS}
 Let $f$ proper, lsc, bounded from below. Assume that $L_{g_\sigma}$ the Lipschitz constant of $\nabla g_\sigma$ and $\beta \in (0,1)$ verify
\begin{equation} \label{eq:L_beta_DRS2}
    \beta(2L_{g_\sigma}^3 - 3L_{g_\sigma}^2 + 1) + (2L_{g_\sigma}^2 + L_{g_\sigma} - 1) < 0 .
\end{equation}
Then, for all $\reglambda > 0$, the iterates $(y_k,z_k,x_k)$
given by the iterative scheme~\eqref{eq:ProxPnP-DRS2} verify
\begin{itemize}
    \item[(i)] $(F^{DR}(x_{k-1},y_k,z_k))$ is non-increasing and converges.
    \item[(ii)] $x_k - x_{k-1} = \beta (y_k - z_k)$ vanishes with rate $\min_{k \leq K} \norm{y_k - z_k} = \mathcal{O}(\frac{1}{\sqrt{K}})$.
     \item[(iii)] If $f$ is real analytic, and $g_\sigma$ is coercive, then the iterates $(x_k)$ converge towards a critical point of $F$. 
\end{itemize}
\end{theorem}
\begin{remark}
    The standard Douglas-Rachford algorithm is for $\beta = 1/2$. The restriction on $L_{g_\sigma}$ is then
    \begin{equation}
    2L_{g_\sigma}^3+L_{g_\sigma}^2+2L_{g_\sigma}-1 < 0 
    \end{equation} 
    which is verified for approximately ${L_{g_\sigma} < 0.376}$. Using a smaller $\beta$ relaxes the constraint on~$L_{g_\sigma}$~until $L_{g_\sigma}<0.5$ when $\beta \to 0$. For example for $\beta=0.25$, the constraint is verified up to  approximately ${L_{g_\sigma} < 0.45}$.
\end{remark}
\begin{remark}
    In addition to deal with non-differentiable data-fidelity terms, the main advantage of this result, compared to \citep[Theorem 4.3]{hurault2022proximal}, is that the condition on $\lambda$ disappears, and the convergence remains true for all regularization parameter $\lambda$.
    However, the restriction on $L_{g_\sigma}$ the Lipschitz constant of $\nabla g_\sigma$ is stronger than just $L_{g_\sigma}< 1$, which should harm the denoising capacity of the denoiser. 
\end{remark}
\begin{proof}
    \textbf{(i)} As explained above, this point represents the main adaptation. Following the initial proof from~\cite[Theorem 1]{themelis2020douglas}, we derive a new sufficient decrease property for $F^{DR}$. We first rewrite the envelope $F^{DR}$ from~\eqref{eq:DRE} as
\begin{equation}
    \begin{split}
        F^{DR}(x,y,z) &= \lambda f(z) +  \phi_\sigma(y) + \langle y-x, y-z \rangle + \frac{1}{2 } \norm{y-z}^2 \\
        &= \lambda f(z) +  \phi_\sigma(y) + \frac{1}{2}\norm{(2y-x)-z}^2 - \frac{1}{2}\norm{x-y}^2 .
\end{split}
\end{equation}
As $z_{k+1} = \prox_{ f}(2y_{k+1}-x_{k})$ \eqref{eq:DRS2_z}, denoting $\forall k > 0$ $F^{DR}_k = F^{DR}(x_{k-1},y_{k},z_{k})$,
\begin{equation}
    \begin{split}
    F^{DR}_{k+1} &= \min_{z \in \mathbb{R}^n}   \lambda f(z) +  \phi_\sigma(y_{k+1}) + \frac{1}{2}\norm{(2y_{k+1}-x_{k})- z }^2 - \frac{1}{2}\norm{x_{k}-y_{k+1}}^2 \\
    &=  \min_{z \in \mathbb{R}^n} \lambda f(z) +  \phi_\sigma(y_{k+1}) + \langle y_{k+1}-x_{k}, y_{k+1}-z \rangle + \frac{1}{2} \norm{y_{k+1}-z}^2 .
    \end{split}    
\end{equation}
The optimality condition for \eqref{eq:DRS2_y} is
\begin{equation}
    x_{k}-y_{k+1} \in \partial \phi_\sigma(y_{k+1}) .
\end{equation}
As $\phi_\sigma$ is differentiable on $\Im(D_\sigma)$ and ${y_{k+1} \in \Im(D_\sigma)}$, 
\begin{equation}\label{eq:x-y}
    x_{k}-y_{k+1}= \nabla \phi_\sigma(y_{k+1}).
\end{equation}
The Douglas–Rachford envelope $F^{DR}$ then writes
\begin{equation}
\label{eq:themelis_DRE2}
\begin{split}
F^{DR}_{k+1} &=   \min_{z \in \mathbb{R}^n} \lambda f(z) +  \phi_\sigma(y_{k+1}) + \langle \nabla \phi_\sigma(y_{k+1}), z-y_{k+1} \rangle + \frac{1}{2} \norm{y_{k+1}-z}^2 .
\end{split}    
\end{equation}
We get, by comparing with the right term when $z=z_k$, 
\begin{equation}
    \begin{split}
        F^{DR}_{k+1} &\leq  \lambda f(z_k) +  \phi_\sigma(y_{k+1}) + \langle \nabla \phi_\sigma(y_{k+1}), z_k-y_{k+1} \rangle + \frac{1}{2} \norm{y_{k+1}-z_k}^2  \\
    &= \phi_\sigma(y_{k+1}) + \langle  \nabla \phi_\sigma(y_{k+1}), y_k-y_{k+1} \rangle  +  \langle  \nabla \phi_\sigma(y_{k+1}), z_k-y_k \rangle \\ &+ \lambda f(z_k) + \frac{1}{2} \norm{y_{k+1}-z_k}^2 .
\end{split}     
\end{equation}
From Proposition~\ref{prop:proxdenoiser}, $\phi_\sigma$ is $M = \frac{L_{g_\sigma}}{L_{g_\sigma}+1}$ weakly convex.
Then with Proposition~\ref{prop:weaklyconvex} (ii) we have 
\begin{equation}
\begin{split}
        \phi_\sigma(y_{k+1}) &+ \langle  \nabla \phi_\sigma(y_{k+1}), y_k-y_{k+1}\rangle \leq \phi_\sigma(y_k) + \frac{M}{2}\norm{y_k-y_{k+1}}^2 
\end{split}
\end{equation}
and 
\begin{equation}
\begin{split}
    F^{DR}_{k+1} &\leq \phi_\sigma(y_k) + \lambda f(z_k) + \langle  \nabla \phi_\sigma(y_{k+1}), z_k-y_k \rangle \\ &+  \frac{M}{2}\norm{y_k-y_{k+1}}^2  + \frac{1}{2} \norm{y_{k+1}-z_k}^2 .
\end{split}
\end{equation}
Moreover,
\begin{equation}
\begin{split}
    \langle  \nabla & \phi_\sigma(y_{k+1}),  z_k-y_k \rangle =  \langle  \nabla \phi_\sigma(y_{k}), z_k-y_k \rangle +  \langle  \nabla \phi_\sigma(y_{k+1}) - \nabla \phi_\sigma(y_{k}), z_k-y_k \rangle
\end{split}
\end{equation}
and using 
\begin{equation}
\begin{split}
    F^{DR}_{k} &= \lambda f(z_k) +  \phi_\sigma(y_{k}) + \langle y_{k}-x_{k-1}, y_{k}-z_k \rangle + \frac{1}{2} \norm{y_{k}-z_k}^2 \\
    &= \lambda f(z_k) +  \phi_\sigma(y_{k}) + \langle \nabla \phi_\sigma(y_k), z_k-y_k \rangle + \frac{1}{2} \norm{y_{k}-z_k}^2
\end{split}
\end{equation}
we get
\begin{equation}
\begin{split}
    F^{DR}_{k+1} &\leq F^{DR}_{k} + \langle  \nabla \phi_\sigma(y_{k+1}) - \nabla \phi_\sigma(y_{k}), z_k-y_k \rangle \\ &- \frac{1}{2} \norm{y_{k}-z_k}^2 + \frac{1}{2} \norm{y_{k+1}-z_k}^2 + \frac{M}{2}\norm{y_k-y_{k+1}}^2 .
\end{split}
\end{equation}
Now using 
\begin{equation}
\begin{split}
    \frac{1}{2} \norm{y_{k+1}-z_k}^2 &= \frac{1}{2}\norm{ y_k-y_{k+1}}^2 + \frac{1}{2}\norm{y_k-z_k}^2 + \langle y_{k+1}-y_k,y_k-z_k\rangle 
\end{split}
\end{equation}
we get
\begin{equation}
\begin{split}
    F^{DR}_{k+1} &\leq F^{DR}_{k} + \langle  (\nabla \phi_\sigma(y_{k+1}) - \nabla \phi_\sigma(y_{k})) - (y_{k+1}-y_k), z_k-y_k \rangle  \\ &+ \frac{1+M}{2}\norm{y_k-y_{k+1}}^2 .
\end{split}
\end{equation}
From~\eqref{eq:DRS2_x} and~\eqref{eq:x-y} 
\begin{equation}
\begin{split}
    z_k-y_k &= \frac{1}{2\beta}(x_k-x_{k-1})
    = \frac{1}{2\beta}(y_{k+1} - y_{k}) + \frac{1}{2\beta}(\nabla \phi_\sigma(y_{k+1}) -\nabla \phi_\sigma(y_{k})) ,
\end{split}
\end{equation}
so that
\begin{equation}
\begin{split}
    F^{DR}_{k+1} - F^{DR}_{k} &\leq \left( 1 +  M - \frac{1}{\beta}\right)\frac{1}{2}\norm{ y_k-y_{k+1}}^2 + \frac{1}{2\beta} \norm{\nabla \phi_\sigma(y_{k+1}) -\nabla \phi_\sigma(y_{k})}^2 .
\end{split}
\end{equation}
Using the fact that $\nabla \phi_\sigma$ is $\frac{L_{g_\sigma}}{1-L_{g_\sigma}}$ Lipschitz on $\Im(D_\sigma)$ and that $\forall k>0, y_k \in \Im(D_\sigma)$,
it simplifies to 
\begin{equation}
\begin{split}
    & F^{DR}_{k+1} - F^{DR}_{k} \leq \frac{1}{2}\left( 1 + M  + \frac{1}{\beta}\left( \frac{L_{g_\sigma}^2}{(1-L_{g_\sigma})^2}-1\right)\right)\norm{ y_k-y_{k+1}}^2 .
\end{split}
\end{equation} 
Recalling that $M=\frac{L_{g_\sigma}}{1+L_{g_\sigma}}$, the condition on the stepsize for $F^{DR}_{k+1} \leq F^{DR}_{k}$ is thus
\begin{equation}
\begin{split}
  &  1 + M  + \frac{1}{\beta}\left( \frac{L_{g_\sigma}^2}{(1-L_{g_\sigma})^2}-1\right) < 0 \\
  &\Leftrightarrow \frac{2L_{g_\sigma}+1}{L_{g_\sigma}+1}  + \frac{1}{\beta}\frac{2L_{g_\sigma}-1}{(1-L_{g_\sigma})^2} < 0 \\
  &\Leftrightarrow \beta(2L_{g_\sigma}^3 - 3L_{g_\sigma}^2 + 1) + (2L_{g_\sigma}^2 + L_{g_\sigma} - 1) < 0 .
\end{split}
\end{equation}
For completeness, recalling that $D_\sigma = \nabla h_\sigma$ with $h_\sigma$ $1-L_{g_\sigma}$ strongly convex,  we have
\begin{equation}
\begin{split}
    \norm{ y_k-y_{k+1}} &= \norm{D_\sigma(x_k)-D_\sigma(x_{k-1})} 
    \geq (1-L_{g_\sigma})  \norm{ x_k-x_{k-1}}
\end{split}
\end{equation}
and the sufficient decrease writes 
\begin{equation}
   F^{DR}_{k} - F^{DR}_{k+1} \geq \delta \norm{x_k-x_{k-1}}^2
\end{equation} 
with 
\begin{equation}
\delta  = \frac{1}{2} (1-L_{g_\sigma})\left(\frac{1}{\beta}\left( 1 - \frac{L_{g_\sigma}^2}{(1-L_{g_\sigma})^2}\right) - (1+M)\right)
\end{equation} 

\textbf{(ii)} This point directly follows from the above sufficient decrease condition.

\textbf{(iii)} For the original DRS algorithm, this point was proven in \citep[Theorem 2]{li2016douglas}. Given the sufficient decrease property, assuming that (a) $F^{DR}$ verifies the KŁ property on $\Im(D_\sigma)$ and (b) that the iterates are bounded, the proof follows equally. With the same arguments as in the proof of Corollary~\ref{cor:PnP-PGD}, both assumptions (a) and (b) are verified when $f$ is real analytic and $g_\sigma$ is coercive.
\end{proof}

\section{Experiments} \label{sec:expes}
In this section, we apply, with the proximal denoiser Prox-DRUNET, the PnP algorithms ProxPnP-PGD~\eqref{eq:ProxPnP-PGD}, 
ProxPnP-$\alpha$PGD~\eqref{eq:ProxPnP-alphaPGD}, ProxPnP-DRSdiff~\eqref{eq:ProxPnP-DRS} (\emph{diff} specifies that this PnP-DRS is dedicated to differentiable data-fidelity terms $f$) and ProxPnP-DRS~\eqref{eq:ProxPnP-DRS2} for deblurring and super-resolution with Gaussian noise.  These four algorithms target critical points of the objective 
\begin{equation}
    F = \lambda f + \phi_\sigma
\end{equation}
with $f$ the data-fidelity term and $\phi_\sigma$ the regularization function detailed in Section~\ref{sec:more_details_reg_proxpnp}.

For both applications, we consider a degraded observation   $y = A x^* + \nu\in \mathbb{R}^m$ of a clean image $x^*\in\mathbb{R}^n$ that is estimated by solving problem~\eqref{eq:pb} with $f(x)=\frac{1}{2}\norm{Ax-y}^2$. Its gradient  $\nabla f(x) = A^T(Ax-y)$ is thus Lipschitz with constant $\norm{A^T A}_S$.  The blur kernels are normalized so that the Lipschitz constant $L_f=1$. We use for evaluation and comparison the 68 images from the CBSD68 dataset, center-cropped to $n=256\times256$ and Gaussian noise with $3$ noise levels $\nu \in\{0.01,0.03,0.05\}$.

For \emph{deblurring}, the degradation operator $A=H$ is a convolution performed with circular boundary conditions. As in~\cite{zhang2017learning,hurault2021gradient,pesquet2021learning,zhang2021plug}, we consider the 8  real-world camera shake kernels of~\cite{levin2009understanding}, the $9 \times9$ uniform kernel and the $25 \times 25$ Gaussian kernel with standard deviation $1.6$.

For single image \emph{super-resolution} (SR), the low-resolution image $y \in \mathbb{R}^m$ is obtained from the high-resolution one $x \in \mathbb{R}^n$ via $y = SHx + \nu$ where $H \in \mathbb{R}^{n \times n}$ is the convolution with anti-aliasing kernel.  The matrix $S$ is the standard $s$-fold downsampling matrix of size $m\times \dime$ and $\dime = s^2 \times m$.     As in~\cite{zhang2021plug}, we evaluate SR performance on 4 isotropic Gaussian blur kernels with  standard deviations $0.7$, $1.2$, $1.6$ and~$2.0$;  and consider downsampled images at scale $s=2$ and $s=3$.

\paragraph{Hyperparameter selection} As explained in Section~\ref{ssec:prox_denoiser}, the proximal denoiser $D_\sigma$ defined in Proposition~\ref{prop:proxdenoiser} is trained following~\cite{hurault2022proximal} with a penalization encouraging $L_g < 1$. Convergence of ProxPnP-PGD, ProxPnP-$\alpha$PGD, ProxPnP-DRSdiff and ProxPnP-DRS are then respectively guaranteed by Corollary~\ref{cor:PnP-PGD}, Corollary~\ref{cor:ProxPnP-alphaPGD}, \cite[Theorem 4.3]{hurault2022proximal} and Theorem~\ref{thm:ProxPnP-DRS}. 
For each algorithm, in Table~\ref{tab:proxpnp_params} we propose default values for the involved hyperparameters. 
Note that we use the same choice of hyperparameters for both deblurring and super-resolution. The $\lambda$ parameter always satisfies the corresponding constraint required for convergence. As each algorithm uses its own set of parameters depending on the constraint on $\lambda$, each algorithm targets critical points of a different functional. 

For ProxPnP-PGD, ProxPnP-$\alpha$PGD and ProxPnP-DRS algorithms, we use the $\gamma$-relaxed version of the denoiser~\eqref{eq:relaxed_proximal_denoiser}. In practice, we found that the same choice of parameters $\gamma$ and $\sigma$ is optimal for both PGD and $\alpha$PGD, with values depending on the amount of noise $\nu$ in the input image. We thus choose $\lambda \in [0,\lambda_{lim}]$ where, following Corollary~\ref{cor:PnP-PGD} and Corollary~\ref{cor:ProxPnP-alphaPGD}, for ProxPnP-PGD $\lambda_{lim}^{\text{PGD}} = \frac{1}{L_f}\frac{\gamma + 2}{\gamma + 1}$ and for ProxPnP-$\alpha$PGD $\lambda_{lim}^{\alpha \text{PGD}} = \frac{1}{L_f}\frac{\gamma + 1}{\gamma} \geq \lambda_{lim}^{\text{PGD}} $. For both $\nu = 0.01$ and $\nu = 0.03$, $\lambda$ is set to its maximal allowed value $\lambda_{lim}$. As $\lambda_{lim}^{\alpha \text{PGD}} \geq \lambda_{lim}^{\text{PGD}}$, ProxPnP-$\alpha$PGD is expected to outperform ProxPnP-PGD at these noise levels. Finally, for ProxPnP-$\alpha$PGD, $\alpha$ is set to its maximum possible value $1 / (\lambda L_f)$. 

For ProxPnP-DRSdiff, $\lambda$ is also set to its maximal possible $\frac{1}{L_f}$ for theoretical convergence \cite[Theorem 4.3]{hurault2022proximal}.

Eventually, for ProxPnP-DRS, Theorem~\ref{thm:ProxPnP-DRS} requires $L_{g_\sigma} < L_{max}(\beta)$ via the constraint~\eqref{eq:L_beta_DRS2}. As $D_{\sigma}$ is trained to ensure $L_{g_\sigma} < 1$, we do not retrain the denoiser for a specific $L_{max}(\beta)$ value, but we use again the $\gamma$-relaxed version of the denoiser $D^\gamma_\sigma$~\eqref{eq:relaxed_proximal_denoiser} with $\gamma = L_{max}(\beta)$. $\gamma \nabla g_\sigma$ is then $L_{max}(\beta)$-Lipschitz and $D^\gamma_\sigma = \prox_{\phi^\gamma_\sigma}$. In practice, we find $\beta=0.25$ to be a good compromise. For this choice of $\beta$, $\eqref{eq:L_beta_DRS2}$ is satisfied for $L_{max}(\beta) \leq 0.45$.

    \begin{table}[!ht]
        \centering \centering \normalsize
        \begin{tabular}{c c c c c }
            $\nu (./255)$ &  & 2.55 & 7.65 & 12.75 \\
            \midrule
             \multirow{3}{*}{PGD}  & 
             $\gamma$ & $0.6$ & $1$ & $1$ \\
             & $\lambda = \frac{\gamma + 2}{\gamma + 1}$  & $1.625$ & $1.5$ & $1.5$ \\
             &  $\sigma/\nu$&$ 1.25 $& $0.75$ & $0.5$ \\
             \midrule
              \multirow{4}{*}{$\alpha$PGD}  & 
             $\gamma$ & $0.6$ & $1$ & $1$ \\
             & $\lambda = \frac{\gamma + 1}{\gamma}$  & $2.66$ & $2$ & $2$ \\
             & $\alpha = \frac{1}{\lambda}$ & $0.37$ & $0.5$ & $0.5$ \\ 
             &  $\sigma/\nu$&$ 1.25 $& $0.75$ & $0.5$ \\
             \midrule
            \multirow{4}{*}{DRS} 
            & $\beta$ & $0.25$ & $0.25$ & $0.25$ \\
            & $\gamma = L_{max}(\beta) $ & $0.45$ & $0.45$ & $0.45$ \\
            & $\lambda$ & $5$ & $1.5$ &  $0.75$ \\
             & $\sigma/\nu$ & $2$ & $1$ & $0.5$ \\
             \midrule
             \multirow{3}{*}{DRSdiff} 
            & $\beta$ & $0.5$ & $0.5$ & $0.5$ \\
            & $\lambda$ & $1.$ & $1.$ &  $1.$ \\
             & $\sigma/\nu$ & $0.75$ & $0.5$ & $0.5$ \\
             \midrule
        \end{tabular}
        \caption{Choice of the different hyperparameters involved for each ProxPnP algorithm. The same set of hyperparameters are used for both debluring and super-resolution experiments.}
        \label{tab:proxpnp_params}
    \end{table}

    \begin{table*}[!ht]
\centering\setlength\tabcolsep{3pt}
\begin{tabular}{c c c c  c c c c c c }
& \multicolumn{3}{c}{Deblurring}& \multicolumn{6}{c}{Super-resolution\vspace*{-2pt}}\\
\cmidrule(lr){2-4}\cmidrule(lr){5-10}
& \multicolumn{3}{c}{}&\multicolumn{3}{c}{scale $s = 2$} & \multicolumn{3}{c}{scale $s = 3$\vspace*{-2pt}} \\
\cmidrule(lr){5-7} \cmidrule(lr){8-10}%
    Noise level $\nu$ & 0.01 & 0.03 & 0.05 & 0.01 & 0.03 & 0.05 & 0.01 & 0.03 & 0.05 \vspace*{-2pt}\\
    \midrule
    IRCNN~\cite{zhang2017learning}& $31.42$ & $28.01$ & $26.40$  & $26.97$ & $25.86$ & $25.45$ & $ 25.60$ & $ 24.72$ & $24.38$\\
    DPIR~\cite{zhang2021plug} & $\mathbf{31.93}$ & $\mathbf{28.30}$ & $\underline{26.82}$&$27.79$ & $26.58$ & $\underline{25.83}$ & $\underline{26.05}$ & $\underline{25.27}$ & $\underline{24.66}$ \\
  GS-PnP~\cite{hurault2021gradient} & $\underline{31.70}$ & $\underline{28.28}$ & $\mathbf{26.86}$  & $27.88$ & $\mathbf{26.81}$ & $\mathbf{26.01}$ & $25.97$ & $\mathbf{25.35}$ & $\mathbf{24.74}$\\
  \midrule
  ProxPnP-PGD & $30.91$ & $27.97$ & $26.66$ & $27.68$ & $26.57$ & $25.81$ & $25.94$ & $25.20$ & $24.62$ \\
  ProxPnP-$\alpha$PGD & $31.55$ & $28.03$ & $26.66$& $\underline{27.92}$ & $\underline{26.61}$ & $25.80$ & $\underline{26.03}$ & $25.26$ & $24.61$ \\
  ProxPnP-DRSdiff & $30.56$ & $27.78$ & $26.61$ & $27.44$ &  $26.58$ & $25.82$ & $25.75$ & $25.19$ & $24.63$ \\
  ProxPnP-DRS & $31.51$ & $28.01$ & $26.63$ & $\mathbf{27.95}$ & $26.58$ & $25.81$ & $\mathbf{26.13}$ & $\underline{25.27}$ & $24.65$\\
\end{tabular}
\caption{PSNR (dB) results on CBSD68 for  deblurring (left) and super-resolution (right). PSNR are averaged over $10$ blur kernels for deblurring (left) and $4$ blur kernels along various scales $s$ for super-resolution (right).
}
\label{tab:whole_results}
\end{table*}

 \paragraph{Numerical performance analysis} We numerically evaluate in Table~\ref{tab:whole_results}, for deblurring and for super-resolution, the PSNR performance of our four ProxPnP algorithms. We give comparisons with the deep state-of-the-art PnP methods IRCNN~\citep{zhang2017learning} and DPIR~\citep{zhang2021plug} which both apply the PnP-HQS algorithm with decreasing stepsize but without convergence guarantees. We also provide comparisons with the GS-PnP (referring to PnP with Gradient Step denoiser) method presented in~\cite{hurault2021gradient}.  

    Observe that, among ProxPnP methods, ProxPnP-DRS and ProxPnP-$\alpha$PGD give the best performance over the variety of kernels and noise levels. Indeed, ProxPnP-DRS convergence is guaranteed for any of~$\lambda$, which can thus be tuned to optimize performance. ProxPnP-DRSdiff is, on the other hand, constrained to $\lambda < 1$.  Similarly, ProxPnP-PGD is constrained to $\lambda < 2$ when ProxPnP-$\alpha$PGD restriction on $\lambda$ is relaxed. When a low amount of regularization (\emph{i.e.} a  large $\lambda$ value) is necessary, an upper bound on $\lambda$ can severely limit the restoration capacity of the algorithm. Indeed, we observe that when the input noise is low ($\nu = 2.55$), ProxPnP-DRSdiff and ProxPnP-PGD perform significantly worse. However, when the input noise is high, a stronger regularization is necessary (\emph{i.e.} a small $\lambda$ value) and all methods perform comparably. 


    We also provide visual comparisons for deblurring 
    in Figure~\ref{fig:proxpnp_deblurring} and super-resolution in Figure~\ref{fig:proxpnp_SR}. Along with the output images, for each ProxPnP algorithm, we plot the evolution of the corresponding provably-decreasing function. We recall in Table~\ref{tab:lyapunov}, for each algorithm, the formula of each function that is proved to decrease and converge along the iterations. We also plot the evolution of the norm of the residuals and of the PSNR along the iterates. These plots empirically confirm the theoretical convergence results.
     Observe that, despite being trained with additional constraints to guarantee convergence, ProxPnP-$\alpha$PGD and ProxPnP-DRS globally compare with the performance of the state-of-the-art DPIR method. 
 
    \begin{table*}[ht]
        \centering
        \begin{tabular}{c|c}
             Algorithm & Decreasing function \\
             \midrule
             ProxPnP-PGD & $F(x) = \lambda f(x) + \phi_\sigma(x)$ \\
             ProxPnP-$\alpha$PGD & $F^{\alpha}(y) = \lambda f(y) + \phi_\sigma(y) + \frac{\alpha}{2}\left(1-\frac1\alpha\right)^2 \norm{y - y_{-}}^2$ \\
             ProxPnP-DRSdiff & 
    $F^{DR}(x,y,z) = \lambda f(y) + \phi_\sigma(z) +  \langle y-x, y-z \rangle + \frac12\norm{y-z}^2$
\\
             ProxPnP-DRS &  $F^{DR}(x,y,z) = \lambda f(z) + \phi_\sigma(y) +  \langle y-x, y-z \rangle + \frac12\norm{y-z}^2$
        \end{tabular}
        \caption{For each ProxPnP algorithm, the corresponding convergence theorem exhibits a function that is proved to decrease along the iterates.}
        \label{tab:lyapunov}
    \end{table*}

\begin{figure*}[!ht] \centering \normalsize
\vspace{-0.5cm}
\captionsetup[subfigure]{justification=centering}
    \begin{subfigure}[b]{.24\linewidth}
        \centering
        \begin{tikzpicture}[spy using outlines={rectangle,blue,magnification=5,size=1.5cm, connect spies}]
        \node {\includegraphics[height=3.cm]{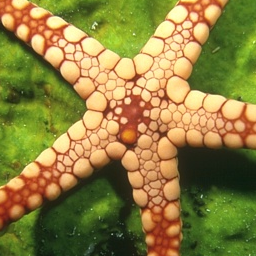}};
        \spy on (0.25,-0.4) in node [left] at (1.5,.75);
        \node at (-1.15,1.15) {\includegraphics[scale=1.2]{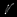}};
        \end{tikzpicture}
        \caption{Clean \\~}
    \end{subfigure}
\begin{subfigure}[b]{.24\linewidth}
        \centering
        \begin{tikzpicture}[spy using outlines={rectangle,blue,magnification=5,size=1.5cm, connect spies}]
        \node {\includegraphics[height=3cm]{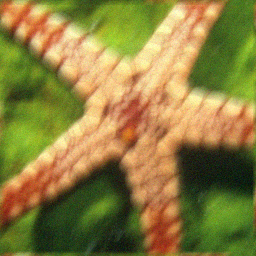}};
        \spy on (0.25,-0.4) in node [left] at  (1.5,.75);
        \end{tikzpicture}
        \caption{Observed \\ ($20.97$dB)}
    \end{subfigure}
\begin{subfigure}[b]{.24\linewidth}
        \centering
        \begin{tikzpicture}[spy using outlines={rectangle,blue,magnification=5,size=1.5cm, connect spies}]
        \node {\includegraphics[height=3cm]{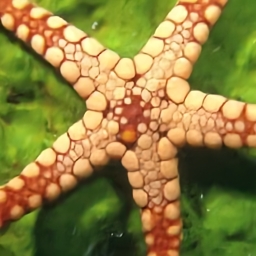}};
        \spy on (0.25,-0.4) in node [left] at  (1.5,.75);
        \end{tikzpicture}
        \caption{IRCNN \\ ($28.66$dB)}
    \end{subfigure}
\begin{subfigure}[b]{.24\linewidth}
        \centering
        \begin{tikzpicture}[spy using outlines={rectangle,blue,magnification=5,size=1.5cm, connect spies}]
        \node {\includegraphics[height=3cm]{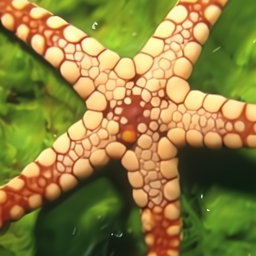}};
        \spy on (0.25,-0.4) in node [left] at  (1.5,.75);
        \end{tikzpicture}
        \caption{DPIR \\ ($29.76$dB)}
    \end{subfigure}
\begin{subfigure}[b]{.24\linewidth}
        \centering
        \begin{tikzpicture}[spy using outlines={rectangle,blue,magnification=5,size=1.5cm, connect spies}]
        \node {\includegraphics[height=3cm]{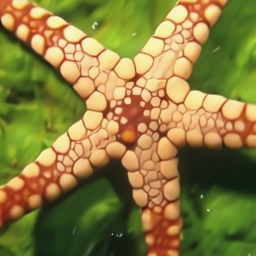}};
        \spy on (0.25,-0.4) in node [left] at (1.5,.75);
        \end{tikzpicture}
        \caption{ProxPnP-PGD ($29.35$dB)}
    \end{subfigure}
\begin{subfigure}[b]{.24\linewidth}
        \centering
        \begin{tikzpicture}[spy using outlines={rectangle,blue,magnification=5,size=1.5cm, connect spies}]
        \node {\includegraphics[height=3cm]{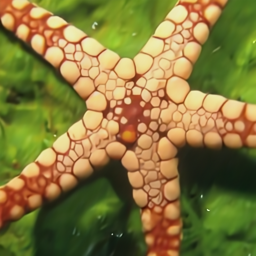}};
        \spy on (0.25,-0.4) in node [left] at (1.5,.75);
        \end{tikzpicture}
        \caption{ProxPnP-$\alpha$PGD ($29.68$dB)}
    \end{subfigure} 
\begin{subfigure}[b]{.24\linewidth}
        \centering
        \begin{tikzpicture}[spy using outlines={rectangle,blue,magnification=5,size=1.5cm, connect spies}]
        \node {\includegraphics[height=3cm]{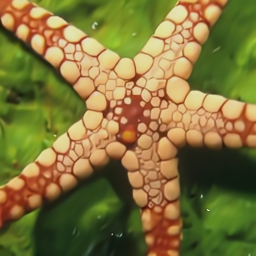}};
        \spy on (0.25,-0.4) in node [left] at (1.5,.75);
        \end{tikzpicture}
        \caption{ProxPnP-DRSdiff ($29.38$dB)}
    \end{subfigure}
    \begin{subfigure}[b]{.24\linewidth}
        \centering
        \begin{tikzpicture}[spy using outlines={rectangle,blue,magnification=5,size=1.5cm, connect spies}]
        \node {\includegraphics[height=3cm]{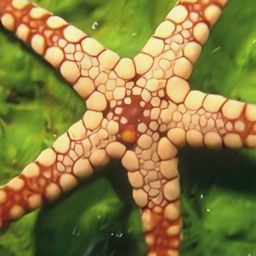}};
        \spy on (0.25,-0.4) in node [left] at (1.5,.75);
        \end{tikzpicture}
        \caption{ProxPnP-DRS ($29.51$dB)}
    \end{subfigure} \\
 \begin{subfigure}[b]{.24\linewidth}
    \centering
      \includegraphics[width=3.25cm]{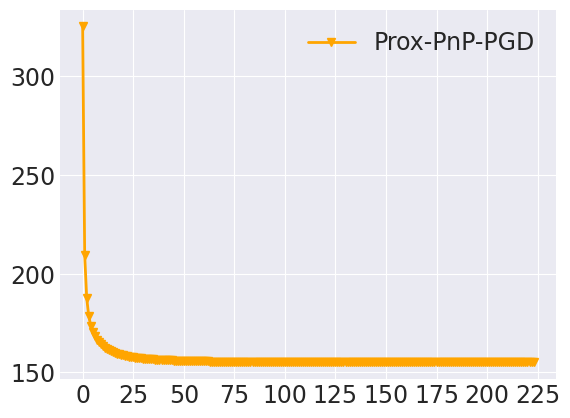}
    \caption{$F(x_k)$}
\end{subfigure}
\begin{subfigure}[b]{.24\linewidth}
    \centering
      \includegraphics[width=3.25cm]{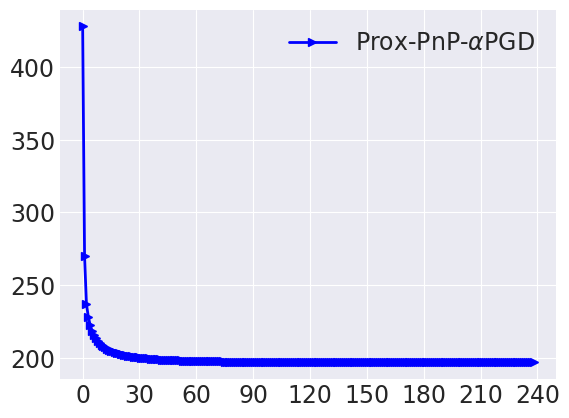}
    \caption{$F^\alpha(x_k)$}
\end{subfigure}
\begin{subfigure}[b]{.24\linewidth}
    \centering
      \includegraphics[width=3.25cm]{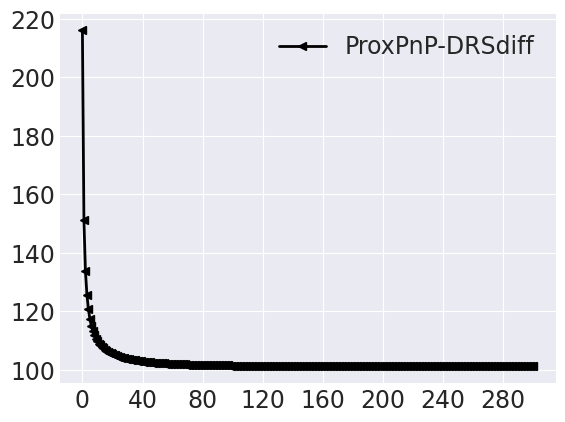}
    \caption{$F^{DR,1}_k$}
\end{subfigure}
\begin{subfigure}[b]{.24\linewidth}
    \centering
      \includegraphics[width=3.25cm]{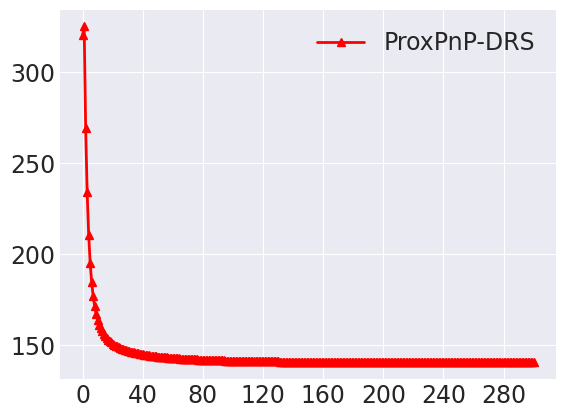}
    \caption{$F^{DR}_k$}
\end{subfigure}
\begin{subfigure}[b]{.4\linewidth}
    \centering
    \includegraphics[width=5cm]{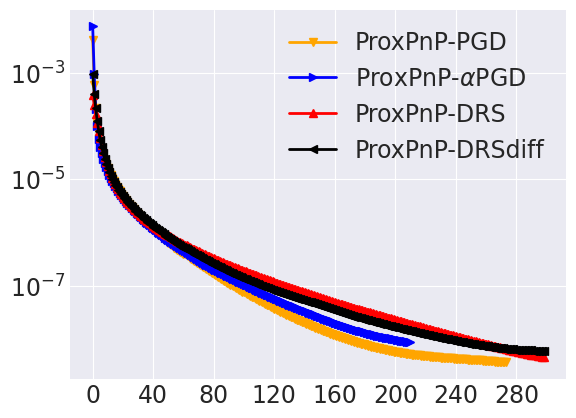}
    \caption{$\gamma_k$ (log scale)}
\end{subfigure}
\begin{subfigure}[b]{.4\linewidth}
    \centering
    \includegraphics[width=4.8cm]{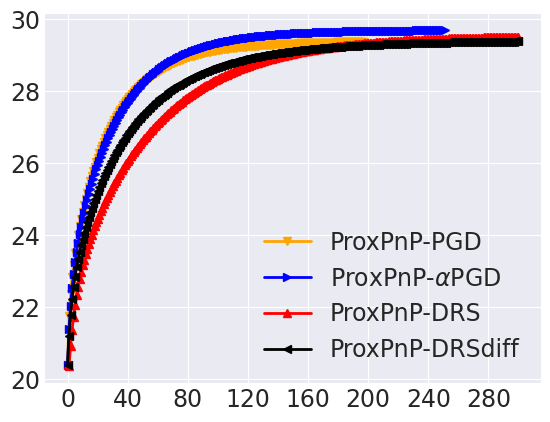}
    \caption{$\text{PSNR}(x_k)$}
\end{subfigure}
\caption{Deblurring with various methods of “starfish” degraded with the indicated blur kernel and input noise level $\nu=0.03$. We also plot for each ProxPnP algorithm the evolution of the respective decreasing Lyapunov functions, the residual ${\gamma_k = \min_{0 \leq i \leq k}\norm{x_{i+1}-x_i}^2}/{{\norm{x_0}^2}}$ and the PSNR.}
\label{fig:proxpnp_deblurring}
\end{figure*}

\begin{figure*}[ht] \centering \normalsize
\captionsetup[subfigure]{justification=centering}
\begin{subfigure}[b]{.24\linewidth}
        \centering
            \begin{tikzpicture}[spy using outlines={rectangle,blue,magnification=5,size=1.5cm, connect spies}]
            \node {\includegraphics[scale=0.35]{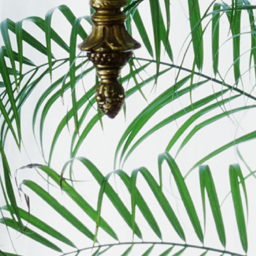}};
            \spy on (-0.4,0.4) in node [left] at (1.55,-0.85);
            \node at (-1.22,1.22) {\includegraphics[scale=0.8]{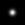}};
            \end{tikzpicture}
        \caption{Clean \\~}
    \end{subfigure}
\begin{subfigure}[b]{.24\linewidth}
        \centering
        \begin{tikzpicture}[spy using outlines={rectangle,blue,magnification=5,size=1.5cm, connect spies}]
        \node {\includegraphics[scale=0.7]{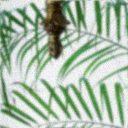}};
        \spy on (-0.4,0.4) in node [left] at (1.55,-0.85);
        \end{tikzpicture}
        \caption{Observed \\ ~}
    \end{subfigure}
\begin{subfigure}[b]{.24\linewidth}
        \centering
        \begin{tikzpicture}[spy using outlines={rectangle,blue,magnification=5,size=1.5cm, connect spies}]
        \node {\includegraphics[scale=0.35]{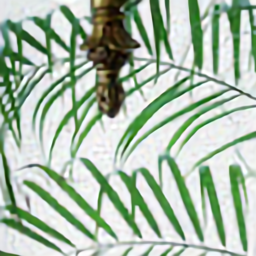}};
        \spy on (-0.4,0.4) in node [left] at (1.55,-0.85);
        \end{tikzpicture}
        \caption{IRCNN \\ ($22.82$dB)}
    \end{subfigure}
\begin{subfigure}[b]{.24\linewidth}
        \centering
        \begin{tikzpicture}[spy using outlines={rectangle,blue,magnification=5,size=1.5cm, connect spies}]
        \node {\includegraphics[scale=0.35]{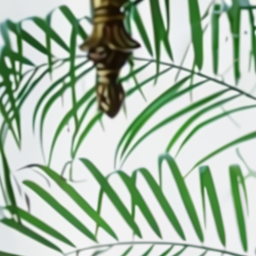}};
        \spy on (-0.4,0.4) in node [left] at (1.55,-0.85);
        \end{tikzpicture}
        \caption{DPIR \\ ($23.97$dB)}
    \end{subfigure}
\begin{subfigure}[b]{.24\linewidth}
        \centering
        \begin{tikzpicture}[spy using outlines={rectangle,blue,magnification=5,size=1.5cm, connect spies}]
        \node {\includegraphics[height=3cm]{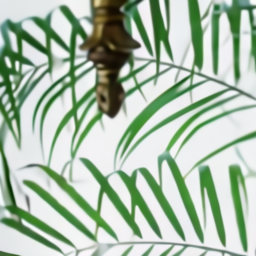}};
        \spy on (-0.4,0.4) in node [left] at (1.55,-0.85);
        \end{tikzpicture}
        \caption{ProxPnP-PGD \\ ($23.98$dB)}
    \end{subfigure}
\begin{subfigure}[b]{.24\linewidth}
        \centering
        \begin{tikzpicture}[spy using outlines={rectangle,blue,magnification=5,size=1.5cm, connect spies}]
        \node {\includegraphics[height=3cm]{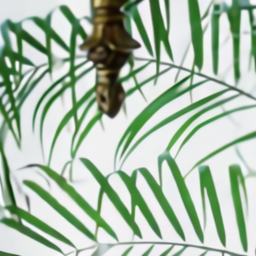}};
        \spy on (-0.4,0.4) in node [left] at (1.55,-0.85);
        \end{tikzpicture}
        \caption{ProxPnP-$\alpha$PGD \\ ($24.22$dB)}
    \end{subfigure}
\begin{subfigure}[b]{.24\linewidth}
        \centering
        \begin{tikzpicture}[spy using outlines={rectangle,blue,magnification=5,size=1.5cm, connect spies}]
        \node {\includegraphics[height=3cm]{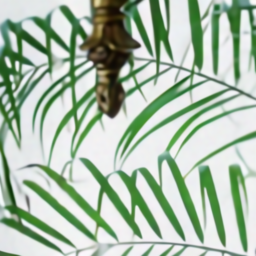}};
        \spy on (-0.4,0.4) in node [left] at (1.55,-0.85);
        \end{tikzpicture}
        \caption{ProxPnP-DRSdiff \\ ($23.96$dB)}
    \end{subfigure}
    \begin{subfigure}[b]{.24\linewidth}
        \centering
        \begin{tikzpicture}[spy using outlines={rectangle,blue,magnification=5,size=1.5cm, connect spies}]
        \node {\includegraphics[height=3cm]{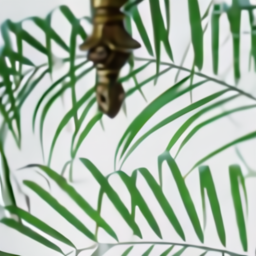}};
        \spy on (-0.4,0.4) in node [left] at (1.55,-0.85);
        \end{tikzpicture}
        \caption{ProxPnP-DRS \\ ($24.17$dB)}
    \end{subfigure}
 \begin{subfigure}[b]{.24\linewidth}
    \centering
    \includegraphics[width=3.25cm]{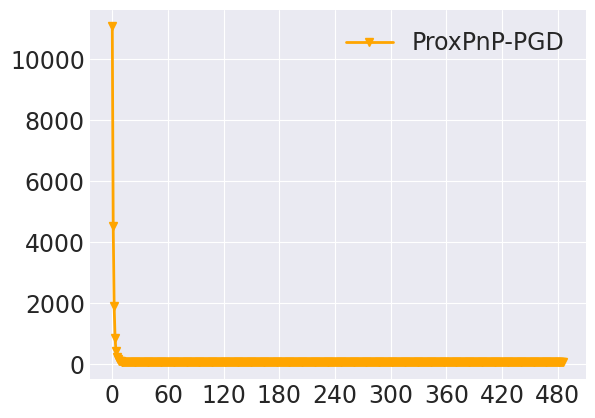}
    \caption{$F(x_k)$}
\end{subfigure}
 \begin{subfigure}[b]{.24\linewidth}
    \centering
    \includegraphics[width=3.25cm]{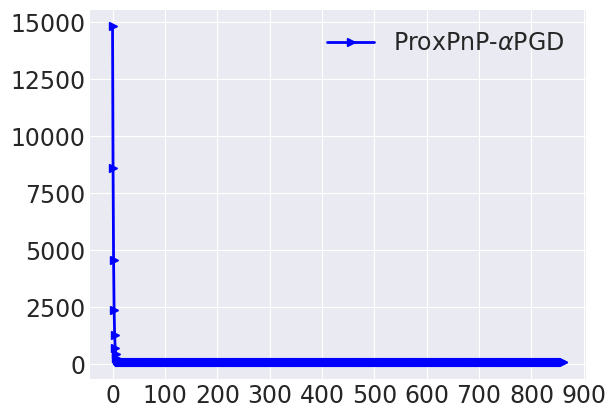}
    \caption{$F^\alpha(x_k)$}
\end{subfigure}
 \begin{subfigure}[b]{.24\linewidth}
    \centering
    \includegraphics[width=3.25cm]{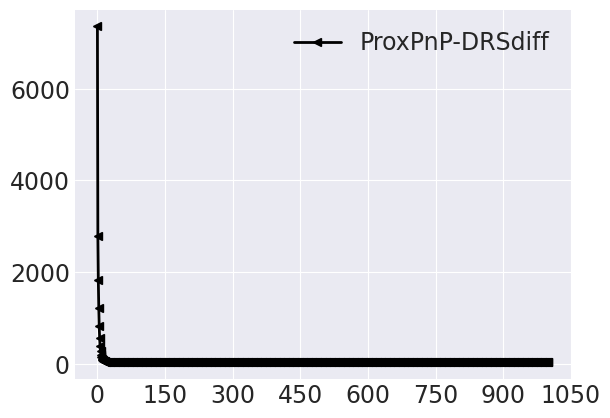}
    \caption{$F(x_k)$}
\end{subfigure}
 \begin{subfigure}[b]{.24\linewidth}
    \centering
    \includegraphics[width=3.25cm]{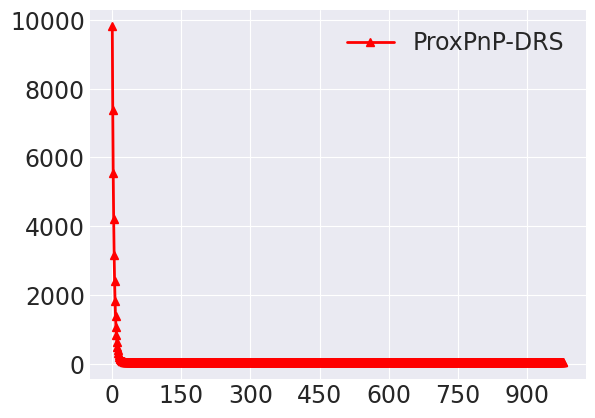}
    \caption{$F(x_k)$}
\end{subfigure}
\begin{subfigure}[b]{.4\linewidth}
    \centering
    \includegraphics[width=5cm]{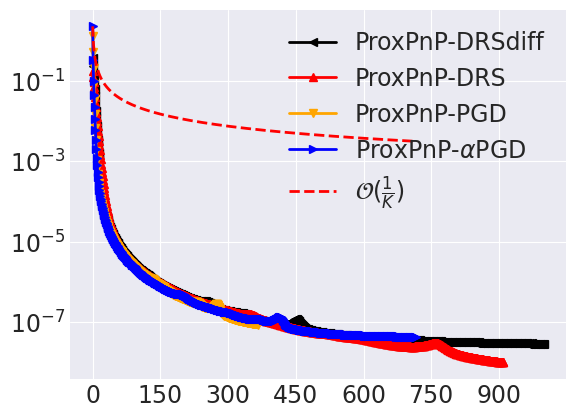}
    \caption{$\gamma_k$ (log scale)}
\end{subfigure}
\begin{subfigure}[b]{0.4\linewidth}
    \centering
    \includegraphics[width=4.8cm]{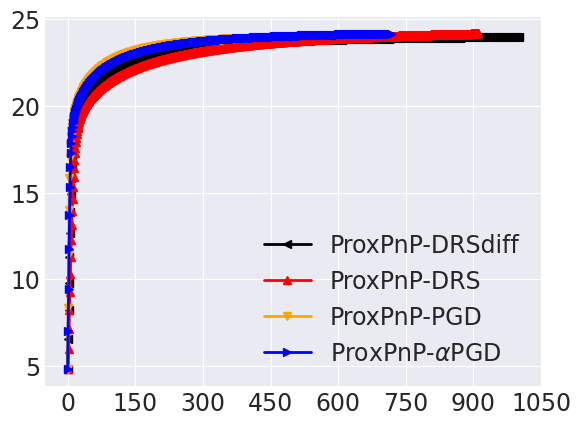}
    \caption{$\text{PSNR}(x_k)$}
\end{subfigure}
\caption{Super-resolution with various methods on “leaves” downsampled by $2$, with the indicated blur kernel and input noise level $\nu=0.03$. We plot the evolution of the respective Lyapunov functions, the residual ${\gamma_k = \min_{0 \leq i \leq k}\norm{x_{i+1}-x_i}^2}/{{\norm{x_0}^2}}$ and the PSNR.}
\label{fig:proxpnp_SR}
\end{figure*}

\section{Conclusion}
In this paper, we propose new convergent plug-and-play algorithms with minimal restrictions on the parameters of the inverse problem. We study two distinct schemes, the first one is a relaxed version of the Proximal Gradient Descent (PGD) algorithm and the second one is the Douglas Rachford Splitting (DRS) scheme. When used with a proximal denoiser, while the original PnP-PGD and PnP-DRS convergence results imposed restrictive conditions on the regularization parameter of the problem, the proposed algorithms converge, with milder conditions, towards stationary points of an explicit functional. Numerical experiments  exhibit  the convergence and the efficiency of the methods on deblurring and super-resolution problems.


A limitation of these approaches is that the $\gamma$-relaxation of the Gradient Step denoiser adds again an extra hyperparameter that needs to be tuned. For instance, PnP-$\alpha$PGD has in total four hyperparameters to tune: $\sigma$, $\lambda$, $\gamma$ and $\alpha$. Even though we proposed default values in Section~\ref{sec:expes}, this could be challenging and time-consuming for the user. 

Different from $\alpha$PGD, another possible modification of the PGD algorithm is the Primal-Dual Davis-Yin (PDDY) algorithm \citep{salim2022dualize}. It is a Primal Dual version of the Davis-Yin splitting method ~\citep{davis2017three}.
As shown by \citet{condat2022randprox} (Algorithm 3), it can be written as a relaxation of the PGD algorihm with a parameter $\alpha \in (0,1)$.
From \cite{condat2022randprox}, \emph{convex} convergence of the ProxPnP version of this algorithm could be established, for any regularization parameter $\lambda$, by choosing $\alpha$ small enough. However, the convergence of the PDDY algorithm in a \emph{nonconvex} setting has not been established and remains an open research avenue for future investigations.

\section*{Acknowledgements}
    This work was funded by the French ministry of research through a CDSN grant of ENS Paris-Saclay. This study has also been carried out with financial support from the French Research Agency through the PostProdLEAP and Mistic projects (ANR-19-CE23-0027-01 and ANR-19-CE40-005).
    A.C. thanks Juan Pablo Contreras for many interesting discussions about nonlinear accelerated descent algorithms.

\clearpage

\begin{appendices}


\section{Primal-Dual derivation of the $\alpha$PGD algorithm}
\label{app:derivation_from_PD}

Suppose we target a solution of the following minimization problem:
\begin{equation}
    \label{eq:PD_pb}
    \min_{x \in \mathbb{R}^n} f(x) + \phi(x)
\end{equation}
for $f$ $L_f$-smooth and strongly convex and $\phi$ convex. The method in~\cite{chambolle2016ergodic} targets a minimizer with a Bregman Primal-Dual algorithm
 \begin{equation} \label{eq:Bregman-Primal-Dual2}
   \left\{\begin{array}{ll} 
        y_{k+1} &= \argmin_y \frac1\sigma  D_{h ^Y}(y,y_k)  + f^*(y) - \langle \bar x_{k},y\rangle \\
        x_{k+1} &= \argmin_x \frac1\tau D_{h ^X}(x,x_k)+ \phi(x) + \langle x,y_{k+1}\rangle \\
        \bar x_{k+1} &= x_{k+1} + \beta (x_{k+1}-x_k) ,
            \end{array}\right.
\end{equation}
where $ D_{h}$ here denotes the Bregman divergence associated to a convex potential~$h$. For convex $f$ and $\phi$, provided that the potentials $h^X$ and $h^Y$ are  $1$-convex with respect to the norm $||.||^2$, convergence of~\eqref{eq:Bregman-Primal-Dual2} towards a solution of \eqref{eq:PD_pb} is ensured as long as $\tau \sigma \norm{K^*K} < 1 $ \citep[Remark 3]{chambolle2016ergodic}.

Notice that $L_f f^*$ is $1$-strongly convex with respect to the norm $||.||^2$. Indeed, $f$ smooth and convex with a $L_f$-Lipschitz gradient implies $f^*$ $1/L_f$ strongly-convex, \emph{i.e.} $L_f f^*$ is $1$-strongly convex. Then we can use $h^Y = L_ff^*$ (and $h^X = \frac{1}{2}\norm{.}^2$) and the algorithm becomes
\begin{equation}
   \left\{\begin{array}{ll} 
         y_{k+1} &= \argmin_y \left(\frac{L_f}{\sigma} +1 \right) f^*(y) - \langle \bar x_{k} + \frac{L_f}{\sigma}\nabla f^*(y^k),y\rangle \\
        x_{k+1} &\in \argmin_x \frac1{2\tau}||x-x^{k}||^2 + \phi(x) + \langle x,y_{k+1}\rangle \\
        \bar x_{k+1} &= x_{k+1} + \beta (x_{k+1}-x_k) .
            \end{array}\right.
\end{equation}
The optimality condition for the first update is
\begin{equation}
    \left(\frac{L_f}{\sigma}+1 \right) \nabla f^*(y_{k+1}) = \bar x_{k} + \frac{L_f}{\sigma}\nabla f^*(y^k) .
\end{equation}
Using the fact that $(\nabla f)^{-1}=\nabla f^*$, with the change of variable $y_k \longleftarrow \nabla f^*(y_{k})$,
the previous algorithm can then be rewritten in a fully primal formulation:
\begin{equation}  \label{eq:VM-Primal}
   \left\{\begin{array}{ll} 
         y_{k+1} &= \frac{\bar x_k + \frac{L_f}{\sigma} y_k}{1 + \frac{L_f}{\sigma}} \\
        x_{k+1} &\in \prox_{\tau \phi}(x_k - \tau  \nabla f(y_{k+1})) \\
        \bar x_{k+1} &= x_{k+1} + \beta (x_{k+1}-x_k) . 
            \end{array}\right.
\end{equation}
For $\alpha = \frac{1}{1 + \frac{L_f}{\sigma}}$, $\beta = 1-\alpha$, the algorithm writes 
\begin{equation}
   \left\{\begin{array}{ll} 
        y_{k+1} &= \alpha \bar x_k + (1-\alpha) y_k\\
        x_{k+1} &\in \prox_{\tau \phi}(x_k - \tau  \nabla f(y_{k+1})) \\
        \bar x_{k+1} &= x_{k+1} + (1-\alpha)(x_{k+1}-x_k) . 
            \end{array}\right.
\end{equation}
or
\begin{equation}
   \left\{\begin{array}{ll} 
  y_{k+1} &= \alpha x_k + (1-\alpha)y_k +  \alpha(1-\alpha)(x_k - x_{k-1}) \\
        x_{k+1} &\in \prox_{\tau \phi}(x_k - \tau \nabla f(y_{k+1}))
            \end{array}\right.
\end{equation}
The latter is equivalent to~\eqref{eq:alphaPGD} with $y_{k}$ in place of $q_{k}$.

\section{Stepsize condition in Theorem~\ref{thm:alphaPGD}.}
\label{app:better_bound}

In Theorem~\ref{thm:alphaPGD}, the stepsize condition 

\begin{equation}
\tau< \min\left(\frac{1}{\alpha L_f},\frac{\alpha}{M}\right)
\end{equation} can be replaced by the slightly better bound
\begin{equation}
     \tau < \min \left( \frac{1}{\alpha L_f}, \frac{2\alpha}{\alpha^3 L_f + (2-\alpha)M}\right).
\end{equation}
which give little numerical gain. For sake of completeness, we develop the proof here.
\begin{proof}
We keep the second term in~\eqref{eq:before_cases} instead of just using its non-negativity.
From the relation
 \begin{equation}
 x_{k+1}-x_{k} = \frac{1}{\alpha}(y_{k+1}-y_k) + \left(1-\frac{1}{\alpha}\right)(y_{k}-y_{k-1}), 
\end{equation}
by convexity of the squared $\ell_2$ norm, for $0 < \alpha < 1$, we have 
\begin{equation}
\begin{split}
\norm{y_{k+1}-y_k}^2 \leq \alpha \norm{x_{k+1}-x_{k}}^2 + (1-\alpha) \norm{y_{k}-y_{k-1}}^2
\end{split}
\end{equation}
and  
\begin{equation}\label{tmp}
\begin{split}
&\norm{x_{k+1}-x_{k}}^2 \geq \frac{1}{\alpha}\norm{y_{k+1}-y_k}^2  +  \left(1-\frac{1}{\alpha}\right) \norm{y_{k}-y_{k-1}}^2.
\end{split}
\end{equation}
which gives finally{\small
\begin{equation}
\begin{split}
  F(y_{k}) - F(y_{k+1}) &\geq  \left(   \alpha\left(   1 - \frac1\alpha\right) \hspace{-3.5pt}\left(   \frac{1}{2\tau}    -  \frac{\alpha L_f}{2}   \right)     -    \frac{\alpha}{2\tau}\left(   1-\frac1\alpha \right)^{ 2} \right)   \norm{y_k -\hspace{-.5pt}y_{k-1}}^2 
  \\
      &+   \left( \frac{1}{2\tau} - \frac{\alpha  L_f}{2} + \frac1\alpha (\frac{1}{2\tau}- \frac{M(2-\alpha)}{ 2 })   \right)
        \norm{y_{k+1}-y_k}^2 \\
    &=  - \frac{1-\alpha}{2\alpha\tau} \left( 1 - \alpha^2 \tau   L_f \right) \norm{y_k-y_{k-1}}^2 \\
    &+   \frac{1}{2\alpha\tau} \left( 1 + \alpha - \alpha^2\tau  L_f - \tau M(2-\alpha) \right)
        \norm{y_{k+1}-y_k}^2 \\
    &= -\delta \norm{y_k - y_{k-1}}^2 + \delta \norm{y_{k+1} - y_{k}}^2 + (\gamma-\delta) \norm{y_{k+1} - y_{k}}^2
    \end{split}
    \end{equation}}
with 
\begin{align}
    \delta &= \frac{1-\alpha}{2\alpha\tau} \left( 1 - \alpha^2 \tau  L_f \right) \\
    \gamma &=  \frac{1}{2\alpha\tau} \left( 1 - \alpha^2\tau  L_f + \alpha - \tau M(2-\alpha) \right).
\end{align}
The condition on the stepsize becomes
\begin{equation}
\begin{split}
    & \gamma - \delta > 0 \\
    &\Leftrightarrow \tau < \frac{2\alpha}{\alpha^3 L_f + (2-\alpha)M},
\end{split}
\end{equation}
and the overall condition is 
\begin{equation}
     \tau < \min \left( \frac{1}{\alpha L_f}, \frac{2\alpha}{\alpha^3 L_f + (2-\alpha)M}\right).
\end{equation}
\end{proof}

\end{appendices}

\bibliography{mybibliography}


\begin{thebibliography}{48}
\ifx \bisbn   \undefined \def \bisbn  #1{ISBN #1}\fi
\ifx \binits  \undefined \def \binits#1{#1}\fi
\ifx \bauthor  \undefined \def \bauthor#1{#1}\fi
\ifx \batitle  \undefined \def \batitle#1{#1}\fi
\ifx \bjtitle  \undefined \def \bjtitle#1{#1}\fi
\ifx \bvolume  \undefined \def \bvolume#1{\textbf{#1}}\fi
\ifx \byear  \undefined \def \byear#1{#1}\fi
\ifx \bissue  \undefined \def \bissue#1{#1}\fi
\ifx \bfpage  \undefined \def \bfpage#1{#1}\fi
\ifx \blpage  \undefined \def \blpage #1{#1}\fi
\ifx \burl  \undefined \def \burl#1{\textsf{#1}}\fi
\ifx \doiurl  \undefined \def \doiurl#1{\url{https://doi.org/#1}}\fi
\ifx \betal  \undefined \def \betal{\textit{et al.}}\fi
\ifx \binstitute  \undefined \def \binstitute#1{#1}\fi
\ifx \binstitutionaled  \undefined \def \binstitutionaled#1{#1}\fi
\ifx \bctitle  \undefined \def \bctitle#1{#1}\fi
\ifx \beditor  \undefined \def \beditor#1{#1}\fi
\ifx \bpublisher  \undefined \def \bpublisher#1{#1}\fi
\ifx \bbtitle  \undefined \def \bbtitle#1{#1}\fi
\ifx \bedition  \undefined \def \bedition#1{#1}\fi
\ifx \bseriesno  \undefined \def \bseriesno#1{#1}\fi
\ifx \blocation  \undefined \def \blocation#1{#1}\fi
\ifx \bsertitle  \undefined \def \bsertitle#1{#1}\fi
\ifx \bsnm \undefined \def \bsnm#1{#1}\fi
\ifx \bsuffix \undefined \def \bsuffix#1{#1}\fi
\ifx \bparticle \undefined \def \bparticle#1{#1}\fi
\ifx \barticle \undefined \def \barticle#1{#1}\fi
\bibcommenthead
\ifx \bconfdate \undefined \def \bconfdate #1{#1}\fi
\ifx \botherref \undefined \def \botherref #1{#1}\fi
\ifx \url \undefined \def \url#1{\textsf{#1}}\fi
\ifx \bchapter \undefined \def \bchapter#1{#1}\fi
\ifx \bbook \undefined \def \bbook#1{#1}\fi
\ifx \bcomment \undefined \def \bcomment#1{#1}\fi
\ifx \oauthor \undefined \def \oauthor#1{#1}\fi
\ifx \citeauthoryear \undefined \def \citeauthoryear#1{#1}\fi
\ifx \endbibitem  \undefined \def \endbibitem {}\fi
\ifx \bconflocation  \undefined \def \bconflocation#1{#1}\fi
\ifx \arxivurl  \undefined \def \arxivurl#1{\textsf{#1}}\fi
\csname PreBibitemsHook\endcsname

\bibitem[\protect\citeauthoryear{Rudin et~al.}{1992}]{ROF}
\begin{barticle}
\bauthor{\bsnm{Rudin}, \binits{L.I.}},
\bauthor{\bsnm{Osher}, \binits{S.}},
\bauthor{\bsnm{Fatemi}, \binits{E.}}:
\batitle{Nonlinear total variation based noise removal algorithms}.
\bjtitle{Phys. D}
\bvolume{60},
\bfpage{259}--\blpage{268}
(\byear{1992})
\end{barticle}
\endbibitem

\bibitem[\protect\citeauthoryear{Tikhonov}{1963}]{tikhonov1963solution}
\begin{bchapter}
\bauthor{\bsnm{Tikhonov}, \binits{A.N.}}:
\bctitle{On the solution of ill-posed problems and the method of
  regularization}.
In: \bbtitle{Doklady Akademii Nauk},
vol. \bseriesno{151},
pp. \bfpage{501}--\blpage{504}
(\byear{1963}).
\bcomment{Russian Academy of Sciences}
\end{bchapter}
\endbibitem

\bibitem[\protect\citeauthoryear{Mallat}{2009}]{mallat2009sparse}
\begin{bbook}
\bauthor{\bsnm{Mallat}, \binits{S.}}:
\bbtitle{A Wavelet Tour of Signal Processing, The Sparse Way},
\bedition{3rd edition} edn.
\bpublisher{Academic Press, Elsevier}, \blocation{???}
(\byear{2009})
\end{bbook}
\endbibitem

\bibitem[\protect\citeauthoryear{Zoran and Weiss}{2011}]{zoran2011learning}
\begin{bchapter}
\bauthor{\bsnm{Zoran}, \binits{D.}},
\bauthor{\bsnm{Weiss}, \binits{Y.}}:
\bctitle{From learning models of natural image patches to whole image
  restoration}.
In: \bbtitle{IEEE ICCV},
pp. \bfpage{479}--\blpage{486}
(\byear{2011})
\end{bchapter}
\endbibitem

\bibitem[\protect\citeauthoryear{Bora et~al.}{2017}]{bora2017compressed}
\begin{botherref}
\oauthor{\bsnm{Bora}, \binits{A.}},
\oauthor{\bsnm{Jalal}, \binits{A.}},
\oauthor{\bsnm{Price}, \binits{E.}},
\oauthor{\bsnm{Dimakis}, \binits{A.G.}}:
Compressed sensing using generative models.
arXiv preprint arXiv:1703.03208
(2017)
\end{botherref}
\endbibitem

\bibitem[\protect\citeauthoryear{Romano et~al.}{2017}]{romano2017little}
\begin{barticle}
\bauthor{\bsnm{Romano}, \binits{Y.}},
\bauthor{\bsnm{Elad}, \binits{M.}},
\bauthor{\bsnm{Milanfar}, \binits{P.}}:
\batitle{The little engine that could: Regularization by denoising (red)}.
\bjtitle{SIAM Journal Imaging Sci.}
\bvolume{10}(\bissue{4}),
\bfpage{1804}--\blpage{1844}
(\byear{2017})
\end{barticle}
\endbibitem

\bibitem[\protect\citeauthoryear{Zhang et~al.}{2021}]{zhang2021plug}
\begin{botherref}
\oauthor{\bsnm{Zhang}, \binits{K.}},
\oauthor{\bsnm{Li}, \binits{Y.}},
\oauthor{\bsnm{Zuo}, \binits{W.}},
\oauthor{\bsnm{Zhang}, \binits{L.}},
\oauthor{\bsnm{Van~Gool}, \binits{L.}},
\oauthor{\bsnm{Timofte}, \binits{R.}}:
Plug-and-play image restoration with deep denoiser prior.
IEEE Transactions on Image Processing
(2021)
\end{botherref}
\endbibitem

\bibitem[\protect\citeauthoryear{Attouch et~al.}{2013}]{attouch2013convergence}
\begin{barticle}
\bauthor{\bsnm{Attouch}, \binits{H.}},
\bauthor{\bsnm{Bolte}, \binits{J.}},
\bauthor{\bsnm{Svaiter}, \binits{B.F.}}:
\batitle{Convergence of descent methods for semi-algebraic and tame problems:
  proximal algorithms, forward--backward splitting, and regularized
  gauss--seidel methods}.
\bjtitle{Math. Programming}
\bvolume{137}(\bissue{1-2}),
\bfpage{91}--\blpage{129}
(\byear{2013})
\end{barticle}
\endbibitem

\bibitem[\protect\citeauthoryear{Themelis and
  Patrinos}{2020}]{themelis2020douglas}
\begin{barticle}
\bauthor{\bsnm{Themelis}, \binits{A.}},
\bauthor{\bsnm{Patrinos}, \binits{P.}}:
\batitle{Douglas--rachford splitting and {ADMM} for nonconvex optimization:
  Tight convergence results}.
\bjtitle{SIAM Journal Optimization}
\bvolume{30}(\bissue{1}),
\bfpage{149}--\blpage{181}
(\byear{2020})
\end{barticle}
\endbibitem

\bibitem[\protect\citeauthoryear{Venkatakrishnan
  et~al.}{2013}]{venkatakrishnan2013plug}
\begin{bchapter}
\bauthor{\bsnm{Venkatakrishnan}, \binits{S.V.}},
\bauthor{\bsnm{Bouman}, \binits{C.A.}},
\bauthor{\bsnm{Wohlberg}, \binits{B.}}:
\bctitle{Plug-and-play priors for model based reconstruction}.
In: \bbtitle{IEEE Glob. Conf. Signal Inf. Process.},
pp. \bfpage{945}--\blpage{948}
(\byear{2013})
\end{bchapter}
\endbibitem

\bibitem[\protect\citeauthoryear{Meinhardt
  et~al.}{2017}]{meinhardt2017learning}
\begin{bchapter}
\bauthor{\bsnm{Meinhardt}, \binits{T.}},
\bauthor{\bsnm{Moller}, \binits{M.}},
\bauthor{\bsnm{Hazirbas}, \binits{C.}},
\bauthor{\bsnm{Cremers}, \binits{D.}}:
\bctitle{Learning proximal operators: Using denoising networks for regularizing
  inverse imaging problems}.
In: \bbtitle{Proceedings of the IEEE International Conference on Computer
  Vision},
pp. \bfpage{1781}--\blpage{1790}
(\byear{2017})
\end{bchapter}
\endbibitem

\bibitem[\protect\citeauthoryear{Zhang et~al.}{2017}]{zhang2017learning}
\begin{bchapter}
\bauthor{\bsnm{Zhang}, \binits{K.}},
\bauthor{\bsnm{Zuo}, \binits{W.}},
\bauthor{\bsnm{Gu}, \binits{S.}},
\bauthor{\bsnm{Zhang}, \binits{L.}}:
\bctitle{Learning deep {CNN} denoiser prior for image restoration}.
In: \bbtitle{IEEE/CVF Conference on Computer Vision and Pattern Recognition},
pp. \bfpage{3929}--\blpage{3938}
(\byear{2017})
\end{bchapter}
\endbibitem

\bibitem[\protect\citeauthoryear{Sun et~al.}{2019}]{sun2019block}
\begin{botherref}
\oauthor{\bsnm{Sun}, \binits{Y.}},
\oauthor{\bsnm{Liu}, \binits{J.}},
\oauthor{\bsnm{Kamilov}, \binits{U.S.}}:
Block coordinate regularization by denoising.
arXiv preprint arXiv:1905.05113
(2019)
\end{botherref}
\endbibitem

\bibitem[\protect\citeauthoryear{Ahmad et~al.}{2020}]{ahmad2020pnp}
\begin{barticle}
\bauthor{\bsnm{Ahmad}, \binits{R.}},
\bauthor{\bsnm{Bouman}, \binits{C.A.}},
\bauthor{\bsnm{Buzzard}, \binits{G.T.}},
\bauthor{\bsnm{Chan}, \binits{S.}},
\bauthor{\bsnm{Liu}, \binits{S.}},
\bauthor{\bsnm{Reehorst}, \binits{E.T.}},
\bauthor{\bsnm{Schniter}, \binits{P.}}:
\batitle{Plug-and-play methods for magnetic resonance imaging: Using denoisers
  for image recovery}.
\bjtitle{IEEE signal processing magazine}
\bvolume{37}(\bissue{1}),
\bfpage{105}--\blpage{116}
(\byear{2020})
\end{barticle}
\endbibitem

\bibitem[\protect\citeauthoryear{Yuan et~al.}{2020}]{Yuan_2020_CVPR}
\begin{bchapter}
\bauthor{\bsnm{Yuan}, \binits{X.}},
\bauthor{\bsnm{Liu}, \binits{Y.}},
\bauthor{\bsnm{Suo}, \binits{J.}},
\bauthor{\bsnm{Dai}, \binits{Q.}}:
\bctitle{Plug-and-play algorithms for large-scale snapshot compressive
  imaging}.
In: \bbtitle{Proceedings of the IEEE/CVF Conference on Computer Vision and
  Pattern Recognition (CVPR)}
(\byear{2020})
\end{bchapter}
\endbibitem

\bibitem[\protect\citeauthoryear{Sun et~al.}{2021}]{sun2021scalable}
\begin{barticle}
\bauthor{\bsnm{Sun}, \binits{Y.}},
\bauthor{\bsnm{Wu}, \binits{Z.}},
\bauthor{\bsnm{Xu}, \binits{X.}},
\bauthor{\bsnm{Wohlberg}, \binits{B.}},
\bauthor{\bsnm{Kamilov}, \binits{U.S.}}:
\batitle{Scalable plug-and-play admm with convergence guarantees}.
\bjtitle{IEEE Transactions Computational Imaging}
\bvolume{7},
\bfpage{849}--\blpage{863}
(\byear{2021})
\end{barticle}
\endbibitem

\bibitem[\protect\citeauthoryear{Ryu et~al.}{2019}]{ryu2019plug}
\begin{bchapter}
\bauthor{\bsnm{Ryu}, \binits{E.}},
\bauthor{\bsnm{Liu}, \binits{J.}},
\bauthor{\bsnm{Wang}, \binits{S.}},
\bauthor{\bsnm{Chen}, \binits{X.}},
\bauthor{\bsnm{Wang}, \binits{Z.}},
\bauthor{\bsnm{Yin}, \binits{W.}}:
\bctitle{Plug-and-play methods provably converge with properly trained
  denoisers}.
In: \bbtitle{International Conference on Machine Learning},
pp. \bfpage{5546}--\blpage{5557}
(\byear{2019})
\end{bchapter}
\endbibitem

\bibitem[\protect\citeauthoryear{Pesquet et~al.}{2021}]{pesquet2021learning}
\begin{barticle}
\bauthor{\bsnm{Pesquet}, \binits{J.-C.}},
\bauthor{\bsnm{Repetti}, \binits{A.}},
\bauthor{\bsnm{Terris}, \binits{M.}},
\bauthor{\bsnm{Wiaux}, \binits{Y.}}:
\batitle{Learning maximally monotone operators for image recovery}.
\bjtitle{SIAM Journal Imaging Sci.}
\bvolume{14}(\bissue{3}),
\bfpage{1206}--\blpage{1237}
(\byear{2021})
\end{barticle}
\endbibitem

\bibitem[\protect\citeauthoryear{Hertrich
  et~al.}{2021}]{hertrich2021convolutional}
\begin{barticle}
\bauthor{\bsnm{Hertrich}, \binits{J.}},
\bauthor{\bsnm{Neumayer}, \binits{S.}},
\bauthor{\bsnm{Steidl}, \binits{G.}}:
\batitle{Convolutional proximal neural networks and plug-and-play algorithms}.
\bjtitle{Linear Algebra and its Applications}
\bvolume{631},
\bfpage{203}--\blpage{234}
(\byear{2021})
\end{barticle}
\endbibitem

\bibitem[\protect\citeauthoryear{Bohra et~al.}{2021}]{bohra2021learning}
\begin{bchapter}
\bauthor{\bsnm{Bohra}, \binits{P.}},
\bauthor{\bsnm{Goujon}, \binits{A.}},
\bauthor{\bsnm{Perdios}, \binits{D.}},
\bauthor{\bsnm{Emery}, \binits{S.}},
\bauthor{\bsnm{Unser}, \binits{M.}}:
\bctitle{Learning lipschitz-controlled activation functions in neural networks
  for plug-and-play image reconstruction methods}.
In: \bbtitle{NeurIPS Workshop on Deep Learning and Inverse Problems}
(\byear{2021})
\end{bchapter}
\endbibitem

\bibitem[\protect\citeauthoryear{Nair and
  Chaudhury}{2022}]{nair2022construction}
\begin{botherref}
\oauthor{\bsnm{Nair}, \binits{P.}},
\oauthor{\bsnm{Chaudhury}, \binits{K.N.}}:
On the construction of averaged deep denoisers for image regularization.
arXiv preprint arXiv:2207.07321
(2022)
\end{botherref}
\endbibitem

\bibitem[\protect\citeauthoryear{Sreehari et~al.}{2016}]{sreehari2016plug}
\begin{barticle}
\bauthor{\bsnm{Sreehari}, \binits{S.}},
\bauthor{\bsnm{Venkatakrishnan}, \binits{S.V.}},
\bauthor{\bsnm{Wohlberg}, \binits{B.}},
\bauthor{\bsnm{Buzzard}, \binits{G.T.}},
\bauthor{\bsnm{Drummy}, \binits{L.F.}},
\bauthor{\bsnm{Simmons}, \binits{J.P.}},
\bauthor{\bsnm{Bouman}, \binits{C.A.}}:
\batitle{Plug-and-play priors for bright field electron tomography and sparse
  interpolation}.
\bjtitle{IEEE Transactions on Computational Imaging}
\bvolume{2}(\bissue{4}),
\bfpage{408}--\blpage{423}
(\byear{2016})
\end{barticle}
\endbibitem

\bibitem[\protect\citeauthoryear{Cohen et~al.}{2021}]{cohen2021has}
\begin{bchapter}
\bauthor{\bsnm{Cohen}, \binits{R.}},
\bauthor{\bsnm{Blau}, \binits{Y.}},
\bauthor{\bsnm{Freedman}, \binits{D.}},
\bauthor{\bsnm{Rivlin}, \binits{E.}}:
\bctitle{It has potential: Gradient-driven denoisers for convergent solutions
  to inverse problems}.
In: \bbtitle{Neural Information Processing Systems},
vol. \bseriesno{34}
(\byear{2021})
\end{bchapter}
\endbibitem

\bibitem[\protect\citeauthoryear{Hurault et~al.}{2022a}]{hurault2021gradient}
\begin{bchapter}
\bauthor{\bsnm{Hurault}, \binits{S.}},
\bauthor{\bsnm{Leclaire}, \binits{A.}},
\bauthor{\bsnm{Papadakis}, \binits{N.}}:
\bctitle{Gradient step denoiser for convergent plug-and-play}.
In: \bbtitle{International Conference on Learning Representations}
(\byear{2022})
\end{bchapter}
\endbibitem

\bibitem[\protect\citeauthoryear{Hurault et~al.}{2022b}]{hurault2022proximal}
\begin{bchapter}
\bauthor{\bsnm{Hurault}, \binits{S.}},
\bauthor{\bsnm{Leclaire}, \binits{A.}},
\bauthor{\bsnm{Papadakis}, \binits{N.}}:
\bctitle{Proximal denoiser for convergent plug-and-play optimization with
  nonconvex regularization}.
In: \bbtitle{International Conference on Machine Learning}
(\byear{2022})
\end{bchapter}
\endbibitem

\bibitem[\protect\citeauthoryear{Hurault et~al.}{2023}]{hurault2023relaxed}
\begin{bchapter}
\bauthor{\bsnm{Hurault}, \binits{S.}},
\bauthor{\bsnm{Chambolle}, \binits{A.}},
\bauthor{\bsnm{Leclaire}, \binits{A.}},
\bauthor{\bsnm{Papadakis}, \binits{N.}}:
\bctitle{A relaxed proximal gradient descent algorithm for convergent
  plug-and-play with proximal denoiser}.
In: \bbtitle{International Conference on Scale Space and Variational Methods in
  Computer Vision},
pp. \bfpage{379}--\blpage{392}
(\byear{2023}).
\bcomment{Springer}
\end{bchapter}
\endbibitem

\bibitem[\protect\citeauthoryear{Gribonval and
  Nikolova}{2020}]{gribonval2020characterization}
\begin{barticle}
\bauthor{\bsnm{Gribonval}, \binits{R.}},
\bauthor{\bsnm{Nikolova}, \binits{M.}}:
\batitle{A characterization of proximity operators}.
\bjtitle{Journal of Mathematical Imaging and Vision}
\bvolume{62}(\bissue{6}),
\bfpage{773}--\blpage{789}
(\byear{2020})
\end{barticle}
\endbibitem

\bibitem[\protect\citeauthoryear{Terris et~al.}{2020}]{terris2020building}
\begin{bchapter}
\bauthor{\bsnm{Terris}, \binits{M.}},
\bauthor{\bsnm{Repetti}, \binits{A.}},
\bauthor{\bsnm{Pesquet}, \binits{J.-C.}},
\bauthor{\bsnm{Wiaux}, \binits{Y.}}:
\bctitle{Building firmly nonexpansive convolutional neural networks}.
In: \bbtitle{IEEE International Conference on Acoustics, Speech, and Signal
  Processing},
pp. \bfpage{8658}--\blpage{8662}
(\byear{2020})
\end{bchapter}
\endbibitem

\bibitem[\protect\citeauthoryear{Rockafellar and Wets}{2009}]{rt1998wets}
\begin{bbook}
\bauthor{\bsnm{Rockafellar}, \binits{R.T.}},
\bauthor{\bsnm{Wets}, \binits{R.J.-B.}}:
\bbtitle{Variational Analysis}
vol. \bseriesno{317}.
\bpublisher{Springer}, \blocation{???}
(\byear{2009})
\end{bbook}
\endbibitem

\bibitem[\protect\citeauthoryear{Attouch et~al.}{2010}]{attouch2010proximal}
\begin{barticle}
\bauthor{\bsnm{Attouch}, \binits{H.}},
\bauthor{\bsnm{Bolte}, \binits{J.}},
\bauthor{\bsnm{Redont}, \binits{P.}},
\bauthor{\bsnm{Soubeyran}, \binits{A.}}:
\batitle{Proximal alternating minimization and projection methods for nonconvex
  problems: An approach based on the kurdyka-{\l}ojasiewicz inequality}.
\bjtitle{Mathematics of operations research}
\bvolume{35}(\bissue{2}),
\bfpage{438}--\blpage{457}
(\byear{2010})
\end{barticle}
\endbibitem

\bibitem[\protect\citeauthoryear{Ochs et~al.}{2014}]{ochs2014ipiano}
\begin{barticle}
\bauthor{\bsnm{Ochs}, \binits{P.}},
\bauthor{\bsnm{Chen}, \binits{Y.}},
\bauthor{\bsnm{Brox}, \binits{T.}},
\bauthor{\bsnm{Pock}, \binits{T.}}:
\batitle{ipiano: Inertial proximal algorithm for nonconvex optimization}.
\bjtitle{SIAM Journal Imaging Sci.}
\bvolume{7}(\bissue{2}),
\bfpage{1388}--\blpage{1419}
(\byear{2014})
\end{barticle}
\endbibitem

\bibitem[\protect\citeauthoryear{Bolte et~al.}{2018}]{bolte2018first}
\begin{barticle}
\bauthor{\bsnm{Bolte}, \binits{J.}},
\bauthor{\bsnm{Sabach}, \binits{S.}},
\bauthor{\bsnm{Teboulle}, \binits{M.}},
\bauthor{\bsnm{Vaisbourd}, \binits{Y.}}:
\batitle{First order methods beyond convexity and lipschitz gradient continuity
  with applications to quadratic inverse problems}.
\bjtitle{SIAM Journal on Optimization}
\bvolume{28}(\bissue{3}),
\bfpage{2131}--\blpage{2151}
(\byear{2018})
\end{barticle}
\endbibitem

\bibitem[\protect\citeauthoryear{Zeng et~al.}{2019}]{zeng2019global}
\begin{bchapter}
\bauthor{\bsnm{Zeng}, \binits{J.}},
\bauthor{\bsnm{Lau}, \binits{T.T.-K.}},
\bauthor{\bsnm{Lin}, \binits{S.}},
\bauthor{\bsnm{Yao}, \binits{Y.}}:
\bctitle{Global convergence of block coordinate descent in deep learning}.
In: \bbtitle{International Conference on Machine Learning},
pp. \bfpage{7313}--\blpage{7323}
(\byear{2019}).
\bcomment{PMLR}
\end{bchapter}
\endbibitem

\bibitem[\protect\citeauthoryear{Krantz and Parks}{2002}]{krantz2002primer}
\begin{bbook}
\bauthor{\bsnm{Krantz}, \binits{S.G.}},
\bauthor{\bsnm{Parks}, \binits{H.R.}}:
\bbtitle{A Primer of Real Analytic Functions},
(\byear{2002})
\end{bbook}
\endbibitem

\bibitem[\protect\citeauthoryear{law Lojasiewicz}{1965}]{law1965ensembles}
\begin{botherref}
\oauthor{\bsnm{Lojasiewicz}, \binits{S.}}:
Ensembles semi-analytiques.
IHES notes,
220
(1965)
\end{botherref}
\endbibitem

\bibitem[\protect\citeauthoryear{Bauschke and
  Combettes}{2011}]{BauschkeCombettes}
\begin{bbook}
\bauthor{\bsnm{Bauschke}, \binits{H.H.}},
\bauthor{\bsnm{Combettes}, \binits{P.L.}}:
\bbtitle{Convex Analysis and Monotone Operator Theory in Hilbert Spaces},
p. \bfpage{468}
(\byear{2011})
\end{bbook}
\endbibitem

\bibitem[\protect\citeauthoryear{Beck and Teboulle}{2009}]{beck2009fast}
\begin{barticle}
\bauthor{\bsnm{Beck}, \binits{A.}},
\bauthor{\bsnm{Teboulle}, \binits{M.}}:
\batitle{Fast gradient-based algorithms for constrained total variation image
  denoising and deblurring problems}.
\bjtitle{IEEE Transactions on Image Processing}
\bvolume{18}(\bissue{11}),
\bfpage{2419}--\blpage{2434}
(\byear{2009})
\end{barticle}
\endbibitem

\bibitem[\protect\citeauthoryear{Tseng}{2008}]{tseng2008accelerated}
\begin{botherref}
\oauthor{\bsnm{Tseng}, \binits{P.}}:
On accelerated proximal gradient methods for convex-concave optimization.
submitted to SIAM Journal Optimization
\textbf{2}(3)
(2008)
\end{botherref}
\endbibitem

\bibitem[\protect\citeauthoryear{Nesterov}{2013}]{nesterov2013gradient}
\begin{barticle}
\bauthor{\bsnm{Nesterov}, \binits{Y.}}:
\batitle{Gradient methods for minimizing composite functions}.
\bjtitle{Mathematical programming}
\bvolume{140}(\bissue{1}),
\bfpage{125}--\blpage{161}
(\byear{2013})
\end{barticle}
\endbibitem

\bibitem[\protect\citeauthoryear{Lan and Zhou}{2018}]{lan2018optimal}
\begin{barticle}
\bauthor{\bsnm{Lan}, \binits{G.}},
\bauthor{\bsnm{Zhou}, \binits{Y.}}:
\batitle{An optimal randomized incremental gradient method}.
\bjtitle{Mathematical programming}
\bvolume{171}(\bissue{1}),
\bfpage{167}--\blpage{215}
(\byear{2018})
\end{barticle}
\endbibitem

\bibitem[\protect\citeauthoryear{Chambolle and Pock}{2011}]{CP11}
\begin{barticle}
\bauthor{\bsnm{Chambolle}, \binits{A.}},
\bauthor{\bsnm{Pock}, \binits{T.}}:
\batitle{A first-order primal-dual algorithm for convex problems with
  applications to imaging}.
\bjtitle{Journal of Math. Imaging and Vision}
\bvolume{40}(\bissue{1}),
\bfpage{120}--\blpage{145}
(\byear{2011})
\end{barticle}
\endbibitem

\bibitem[\protect\citeauthoryear{Chambolle and
  Pock}{2016}]{chambolle2016ergodic}
\begin{barticle}
\bauthor{\bsnm{Chambolle}, \binits{A.}},
\bauthor{\bsnm{Pock}, \binits{T.}}:
\batitle{On the ergodic convergence rates of a first-order primal--dual
  algorithm}.
\bjtitle{Mathematical Programming}
\bvolume{159}(\bissue{1}),
\bfpage{253}--\blpage{287}
(\byear{2016})
\end{barticle}
\endbibitem

\bibitem[\protect\citeauthoryear{Bauschke and
  Combettes}{2011}]{bauschke2011convex}
\begin{bbook}
\bauthor{\bsnm{Bauschke}, \binits{H.H.}},
\bauthor{\bsnm{Combettes}, \binits{P.L.}}:
\bbtitle{Convex Analysis and Monotone Operator Theory in Hilbert Spaces}
vol. \bseriesno{408}.
\bpublisher{Springer}, \blocation{???}
(\byear{2011})
\end{bbook}
\endbibitem

\bibitem[\protect\citeauthoryear{Li and Pong}{2016}]{li2016douglas}
\begin{barticle}
\bauthor{\bsnm{Li}, \binits{G.}},
\bauthor{\bsnm{Pong}, \binits{T.K.}}:
\batitle{Douglas--rachford splitting for nonconvex optimization with
  application to nonconvex feasibility problems}.
\bjtitle{Math. Progr.}
\bvolume{159},
\bfpage{371}--\blpage{401}
(\byear{2016})
\end{barticle}
\endbibitem

\bibitem[\protect\citeauthoryear{Levin et~al.}{2009}]{levin2009understanding}
\begin{bchapter}
\bauthor{\bsnm{Levin}, \binits{A.}},
\bauthor{\bsnm{Weiss}, \binits{Y.}},
\bauthor{\bsnm{Durand}, \binits{F.}},
\bauthor{\bsnm{Freeman}, \binits{W.T.}}:
\bctitle{Understanding and evaluating blind deconvolution algorithms}.
In: \bbtitle{IEEE/CVF Conference on Computer Vision and Pattern Recognition},
pp. \bfpage{1964}--\blpage{1971}
(\byear{2009})
\end{bchapter}
\endbibitem

\bibitem[\protect\citeauthoryear{Salim et~al.}{2022}]{salim2022dualize}
\begin{barticle}
\bauthor{\bsnm{Salim}, \binits{A.}},
\bauthor{\bsnm{Condat}, \binits{L.}},
\bauthor{\bsnm{Mishchenko}, \binits{K.}},
\bauthor{\bsnm{Richt{\'a}rik}, \binits{P.}}:
\batitle{Dualize, split, randomize: Toward fast nonsmooth optimization
  algorithms}.
\bjtitle{Journal of Optimization Theory and Applications}
\bvolume{195}(\bissue{1}),
\bfpage{102}--\blpage{130}
(\byear{2022})
\end{barticle}
\endbibitem

\bibitem[\protect\citeauthoryear{Davis and Yin}{2017}]{davis2017three}
\begin{barticle}
\bauthor{\bsnm{Davis}, \binits{D.}},
\bauthor{\bsnm{Yin}, \binits{W.}}:
\batitle{A three-operator splitting scheme and its optimization applications}.
\bjtitle{Set-valued and variational analysis}
\bvolume{25},
\bfpage{829}--\blpage{858}
(\byear{2017})
\end{barticle}
\endbibitem

\bibitem[\protect\citeauthoryear{Condat and
  Richt{\'a}rik}{2022}]{condat2022randprox}
\begin{botherref}
\oauthor{\bsnm{Condat}, \binits{L.}},
\oauthor{\bsnm{Richt{\'a}rik}, \binits{P.}}:
Randprox: Primal-dual optimization algorithms with randomized proximal updates.
arXiv preprint arXiv:2207.12891
(2022)
\end{botherref}
\endbibitem

\end{thebibliography}

\end{document}